\documentclass[11pt, reqno]{amsart}

\usepackage{amsmath,amssymb,amsfonts,amsthm}
\newtheorem{theorem}{Theorem}
\newtheorem*{theorem*}{Theorem}
\newtheorem{definition}{Definition}[section]
\newtheorem{proposition}[definition]{Proposition}

\newtheorem{remark}[definition]{Remark}
\newtheorem{lemma}[definition]{Lemma}
\newtheorem*{lemma*}{Lemma}
\newtheorem{corollary}[definition]{Corollary}
\newenvironment{preuve}{\noindent {\it Proof}}{\hfill$\square$}

\usepackage{color}                    
\usepackage[colorlinks,citecolor=blue]{hyperref}
\definecolor{uququq}{rgb}{0.25,0.25,0.25}

\usepackage{graphics}
\usepackage{pgf,tikz}
\usetikzlibrary{arrows}
\usetikzlibrary{calc}
\usepackage[utf8]{inputenc}
\newcommand{\al}{\alpha }
\newcommand{\be}{\beta }
\newcommand{\ga}{\gamma }
\newcommand\de{\delta}

\newcommand\ep{\varepsilon}
\newcommand{\e}{\varepsilon}

\newcommand{\pp}{{\mathbb{P}}}

\newcommand{\E}{{\mathbb{E}}}

\newcommand{\indiq}{{\bf 1}}

\newcommand{\rr}{{\mathbb{R}}} 
\newcommand{\RR}{{\mathbb{R}}} 
\newcommand{\R}{{\mathbb{R}}}
\newcommand{\nn}{{\mathbb{N}}} 

\newcommand{\N}{{\mathbb{N}}}

\newcommand{\cF}{{\mathcal F}}

\newcommand{\cH}{{\mathcal H}}

\newcommand{\cJ}{{\mathcal J}}

\newcommand{\cL}{{\mathcal L}}
\newcommand{\cM}{{\mathcal M}}

\newcommand{\cP}{{\mathcal P}}

\newcommand{\cS}{{\mathcal S}}

\newcommand\calX{\mathcal{X}}

\newcommand\calL{\mathcal{L}}
\newcommand\calP{\mathcal{P}}

\newcommand\rd{\mathrm{d}}

\newcommand{\ds}{\displaystyle}
\newcommand{\ov}{\overline}

\newcommand{\wt}{\widetilde}
\newcommand{\wh}{\widehat}

\def \( {\left(}
\def \) {\right)}
\def \[ {\left\lbrack}
\def \] {\right\rbrack}
\def \lba {\left|}
\def \rba {\right|}
\def \llb {\left\lbrace}
\def \rrb {\right\rbrace}
\def \< {\left\langle}
\def \> {\right\rangle}
\def\sk{\smallskip}
\def\mk{\medskip}

\newcommand\zu{[0,1]}

\newcommand\wn{\widetilde N}

\newcommand{\sm}{{s-}}
\newcommand{\um}{{u-}}

\newcommand\dimh{\dim_{\cH}}

\newcommand{\ala}{\nonumber \\}

\begin{document}

\title[]{Multifractality of jump diffusion processes}
\author[]{Xiaochuan Yang}
%\address{Université Paris-Est\\ LAMA(UMR8050), UPEMLV, UPEC, CNRS \\ F-94010 Créteil \\France}
\address{Universit\'e Paris-Est, LAMA(UMR8050), UPEMLV, UPEC, CNRS, F-94010 Cr\'eteil, France.}
\address{Dept. Statistics \& Probability, Michigan State University, 48824 East Lansing, MI, USA}
%\email[]{xiaochuan.yang@u-pec.fr}
\email[]{yangxi43@stt.msu.edu, xiaochuan.j.yang@gmail.com}
\date{\today}

\begin{abstract} We study the  local regularity and multifractal nature   of the sample paths of jump diffusion processes, which are solutions to a class of stochastic differential equations with jumps. This article extends the recent work of Barral {\it et al.} who constructed a pure jump monotone Markov process with random multifractal spectrum. The class of processes studied here is much larger and exhibits novel features on the extreme values of the  spectrum. This class includes  Bass' stable-like processes and non-degenerate stable-driven SDEs.
  %We explain, especially by computing their tangent processes at all points,  why these jump diffusion processes locally look  like a Lévy process. 
%On étudie la regularité locale  et la nature multifractale des trajectoires d'une classe de diffusion à sauts, solutions d'EDS  \begin{equation*}  M_t = \int_0^t  \sigma(M_{s-})\,dB_s + \int_0^t  b(M_s)\,ds + \int_0^t \int_{[-1,1]\backslash\{0\}}  G(M_{s-},z)\, \widetilde{N} (ds,dz). \end{equation*}  Cette classe contient les processus stable-like au sens de Bass.  On montre que  $M$ a des propriétés multifractales remarquables, et que son spectre multifractal est non seulement aléatoire, mais aussi dépend du temps.  On explique, en calculant ses processus tangents,  pourquoi les diffusion à sauts comportent  localement comme un processus de Lévy. 
\sk

Nous \'etudions la r\'egularit\'e locale et la nature multifractale des trajectoires de diffusion \`a sauts, qui sont solutions d' une classe d'\'equations stochastiques \`a sauts.  Cet article prolonge et étend substantiellement le travail r\'ecent de Barral {\it et al.} qui ont construit un processus de Markov de sauts purs avec un spectre multifractal al\'eatoire.  La classe  consid\'er\'ee est beaucoup plus large et présente de nouveaux  ph\'enom\`enes multifractals notamment sur les valeurs extr\^emes du spectre.  Cette classe comprend les processus de type stable au sens de Bass et  des EDS non d\'eg\'en\'er\'ees  guid\'ees par un processus stable. 
\end{abstract}

\thanks{This work was supported by grants from Région Ile-de-France.}
 
\maketitle

\section{Introduction and main results}

This article concerns the pointwise regularity of the sample paths of Markov processes.  As two fundamental steps, in \cite{jaffard1999levy, barral2010increasing} are considered the class of L\'evy processes and a specific example of pure jump increasing Markov process. 

We investigate here the multifractal structure of a quite general class of one-dimensional Markov processes defined by stochastic differential equations with jumps, called jump diffusions :
\begin{align}\label{jd}
M_t =& M_0 + \int_0^t \sigma(M_{s-}) \,{\rm d} B_s + \int_0^t b(M_s) \,{\rm d}s \nonumber \\
 &+ 
 \int_0^t\int_{|z|\le 1} G(M_{s-},z) \,\wt N({\rm d}s, {\rm d}z) + \int_0^t \int_{|z|>1} F(M_{s-},z)\,N(\rd s, \rd z),
\end{align}
where $B$ is a standard Brownian motion, $N$ is a Poisson random measure with intensity measure $\rd t\, \pi(\rd z)$, $\wt N$ is the associated compensated Poisson measure, $\sigma, b, G, F$ are real valued functions satisfying conditions stated below, see \cite{ikeda1989, situ2005, applebaum2009} for many details on SDEs with jumps.

\smallskip

This work is in the line of a large and recent literature investigating the multifractal nature of "L\'evy-like" processes \cite{jaffard1999levy, barral2007timechange,durand2009singularity, barral2010increasing,durand2012levyfields,balanca2014levy},
and in particular, aims to generalize the recent work \cite{jaffard1999levy} of Jaffard on L\'evy processes to a larger class of Markov processes.  As an attempt beyond the scope of L\'evy processes,  Barral-Fournier-Jaffard-Seuret \cite{barral2010increasing} constructed a specific example of pure jump {\em increasing} Markov process with a random multifractal spectrum.  Namely, they investigated the spectrum of a Markov process characterized by the following jump measure: 
\begin{align*}
\nu_\ga(y,du)= \ga(y)u^{-1-\ga(y)}\indiq_{[0,1]}\rd u
\end{align*} 
where $\ga: \rr\mapsto (0,1)$ is supposed to be Lipschitz-continuous and strictly increasing. Clearly, the specific structure of this process (and in particular the monotonicity of sample paths) simplifies greatly the study of its regularity, and as a consequence, the present work requires more technicality and additional tools.

Let us comment on the major contributions of the present work relative to what has already been achieved in the area and in particular in \cite{barral2010increasing}.
\begin{itemize}
\item  It is the first that treats a quite general class of Markov processes beyond L\'evy processes -- SDE with jumps. The class of processes studied is much larger, and not anymore restricted to a special case of positive increasing jump diffusions with stable-like index in $(0,1)$. In particular, Bass' variable order stable-like processes (with stable-like index in $(0,2)$, thus of infinite variation) and the class of non-degenerate stable-driven SDEs are included, see Example \ref{example1} and \ref{example2}.
\item  A slicing argument is developed to give some technical increments estimates. This new argument does not rely on the monotonicity of the sample paths, thus is applicable to more general SDEs, see Section \ref{upj}.
\item There is a novel discussion on the extreme value of the spectrum, the latter presenting a behavior more complex than the one observed in \cite{barral2010increasing}, see Section \ref{secspec}.
\end{itemize}

\medskip

\subsection{Recalls on multifractal analysis}
Multifractal properties are now identified  as important features of sample paths of stochastic processes.  
The variation of the regularity of stochastic processes has been observed considerably since mid-70's, e.g. fast and slow points of (fractional) Brownian motion \cite{orey1974fast,perkins1983slow, khoshnevisan2000fastforfBm},  "L\'evy-like" processes \cite{jaffard1999levy,barral2007timechange,durand2009singularity, barral2010increasing, durand2012levyfields, balanca2014levy}, SPDEs \cite{perkins1998superbrownian,mytnik2015superstable, balanca2016}, among many other examples. Multifractal analysis  turns out to be a relevant approach to draw a global picture of the distribution of singularities.

Let us recall some relevant notions in our study. The regularity exponent we consider is the pointwise H\"older exponent.  We use Hausdorff dimension, denoted by $\dimh$ with convention $\dimh\emptyset = -\infty$, to study the singularity sets (iso-H\"older sets defined below). See \cite{falconer2003} for more on dimensions and \cite{jaffard2004survey} for many aspects of multifractal analysis.

%%%%%%%%%%%%%%%%%%%%%%%%%%%%%%%%%%%%%%%%%%%%%%%
\begin{definition}
Let $f\in L^{\infty}_{loc}(\rr)$, $x_0\in \R$ and $h \in \R^+\setminus \N^*$. We say that $f$ belongs to $C^{h}(x_0)$  if there exist two positive constants $C$, $\eta$, a polynomial $P$ with degree less than $h$, such that $|f(x)-P(x-x_0)|\le C|x-x_0|^{h}$ when $|x-x_0|<\eta.$ The pointwise H\"older exponent of $f$ at $x_0$ is  $$ H_f(x_0)= \sup\llb h\ge 0: f\in C^{h}(x_0)\rrb . $$
\end{definition}
%%%%%%%%%%%%%%%%%%%%%%%%%%%%%%%%%%%%%%%%%%%%%%%

\smallskip

%%%%%%%%%%%%%%%%%%%%%%%%%%%%%%%%%%%%%%%%%%
\begin{definition}
Let $f\in L^{\infty}_{loc}(\rr)$. For $h\geq 0$, the iso-Hölder set of order $h$ is 
\begin{align*}
E_f(h)= \llb x\in\R : H_f(x)=h \rrb
\end{align*}
 and the multifractal spectrum of  $f$ is the mapping $D_f : \R^+ \to [0,1]\cup\{-\infty\}$ defined by
 \begin{align*}
 h\mapsto D_f(h)= \dim_{\cH} E_f(h)
 \end{align*}
We also define, for any open set $A\subset\rr^+$,  the local spectrum  of $f$ on $A$ as 
\begin{equation} \label{lsA}
D_f(A,h) = \dim_{\cH}(A\cap E_f(h)). 
\end{equation}
\end{definition}
%%%%%%%%%%%%%%%%%%%%%%%%%%%%%%%%%%%%%%%%%%%

Apart from \cite{barral2010increasing},  the aforementioned examples have \emph{homogeneous} multifractal spectra, that is there is no dependency on the region where the spectra are computed: $D_f(\rr^+,h) = D_f(A,h)$, for all open sets $A \in \rr^+$.  The example constructed by Barral {\it et al.} in \cite{barral2010increasing} has multifractal characteristics that change as time passes. It is thus relevant to consider the pointwise multifractal spectrum at a given point. Other examples with varying pointwise spectrum are studied in  \cite{durand2008treeindexedchain, barral2013localformalism,balanca2014levy}.

\begin{definition} Let $f\in L^{\infty}_{loc}(\rr)$, $t_0 \in \rr^+$, and let $B_r$ be an open interval centered at $t_0$ with radius $r>0$. The pointwise multifractal spectrum of $f$ at $t_0$ is the mapping
\begin{equation}
\forall\, h \ge 0, \ \ \ \ D_f (t_0,h) = \lim_{r \to 0 } D_f(B_r,h).\nonumber
\end{equation}
\end{definition}  

As was pointed out in \cite[Definition 3, Lemma 4]{barral2010increasing} (see also  the remark below Definition 4 and Proposition 2 in  \cite{barral2013localformalism}), the pointwise spectrum is well-defined in the sense that it does not depend on the sequence of open intervals chosen, and the local spectrum $D_f(A,h)$ on any open set  $A$ can be completely recovered from the pointwise spectrum. More precisely, for any open set $A\subset\rr_+$ and any $h\ge 0$, we have 
\begin{equation}\label{frompointtolocal}
D_f(A,h)= \sup_{t\in A} D_f(t,h).
\end{equation}
  Thus these two types of results are equivalent and one can pass from one to the other easily. 

Below is a simple fact that will be useful in the derivation of a pointwise spectrum. 

\begin{lemma}\cite[Corollary 1]{barral2013localformalism} \label{scs} The mapping $t\mapsto D_f(t,h)$ is upper semi-continuous.
\end{lemma}

Let us end this subsection with Jaffard's Theorem on the multifractal nature of L\'evy processes \cite{jaffard1999levy}. Recall that \cite[page 126]{applebaum2009} any one dimensional L\'evy process can be written as 
$$Z_t = aB(t) + bt +  \int_0^t\int_{|z|\le 1} z \,\wt N(\rd s,\rd z) + \int_0^t\int_{|z|>1} z \, N(\rd s,\rd z)$$ 
with triplet $(a,b,\pi(\rd z))$ where $\pi$ is the L\'evy (intensity) measure of the Poisson measure $N$ satisfying $\int 1\wedge |z|^2 \,\pi(\rd z)<+\infty$.
Define the Blumenthal-Getoor upper index
\begin{align*}
\be_\pi = \inf\llb \ga\ge 0 :  \int_{|z|\le 1} |z|^\ga \,\pi(\rd z) <+\infty  \rrb .
\end{align*}
Now let $Z$ be a L\'evy process with non trivial Brownian component ($a\neq 0$) and with index $\be_\pi\in(0,2)$. Jaffard  \cite{jaffard1999levy} established: almost surely,  at every $t_0>0$,  the sample path of $Z$ has the (deterministic) pointwise spectrum 
\begin{equation}\label{jaffard}
D_Z(t_0,h)  = D_Z(h) = \begin{cases} \  \beta_\pi h & \mbox { if } h\in [0,1/2),\\  \ \  1 & \mbox{ if } h =1/2,  \\ -\infty  & \mbox { if } h>1/2.\end{cases}
\end{equation}
In particular, L\'evy processes are homogeneously multifractal.

\subsection{Assumptions in this work}
The equation \eqref{jd} is a very general formulation, as it is shown by Çinlar and Jacod \cite{cinlar1981semimartingaleMarkov}. We need some conditions on the coefficients and on the Poisson measure so that we can perform a complete multifractal analysis.

Let us start with the setting.  Let $B$ be a standard Brownian motion and $N(\rd t, \rd z)$ be a Poisson random measure with intensity measure $\rd t \, \pi(\rd z)$, $M_0$ be a random variable with distribution $\mu$ defined independently on the probability space $(\Omega, \cF, \pp)$.  Let $(\cF_t)$ be the minimal augmented filtration associated with $B$, $N$ and $M_0$.  Our process $M$ is the strong solution (see \cite[page 76]{situ2005} for a definition) to  \eqref{jd}.
To make the presentation transparent, we assume throughout the paper
\begin{align*}
\ \ \pi(\rd z) = \rd z/z^2.
\end{align*}
From the definition of the Blumenthal-Getoor index we see that with this measure $\be_\pi=1$. Let us briefly mention the modifications needed to treat more general $\pi(\rd z)$.  Firstly, the auxiliary limsup sets $A_\de$ (defined in Section \ref{sectionproofj}) are in criticality at $\de=\be_\pi$ in the sense that $A_\de\supset \RR^+$ almost surely if $\de<\be_\pi$, and $\RR^+\setminus A_\de\neq\emptyset$ with positive probability if $\de>\be_\pi$. Secondly, one needs to extend Barral-Seuret's localized ubiquity theorem (Theorem \ref{bs} below) to general intensity measures of the underlying Poisson point process. In the author's dissertation \cite[Chapter 3]{yang2016thesis}, the mentioned theorem is generalized to allow for instance singular measures. The first two conditions are the usual growth condition and local Lipschitz condition for the existence and uniqueness of a solution, see  \cite{fu2010positiveSDE, li2011strongsolution}. 

\sk
\noindent (H1) There is a finite constant $K$ such that
\begin{align*}
\forall x\in\rr,\ \ \ \sigma(x)^2 + b(x)^2 + \int_{|z|\le 1} G(x,z)^2 \,\pi(\rd z)  < K(1+x^2).
\end{align*}

\sk
\noindent(H2) For all $m\in\nn^*$, there is a finite  constant $c_m$ such that 
\begin{align*}
\forall\,|x|, |y|\le m, \ \ \  |\sigma(x)-\sigma(y)| + |b(x)-b(y)| \le c_m |x-y|.
\end{align*}

Note that no condition is needed for $F$, because $\pi(\{z: |z|>1 \})<+\infty$, see  \cite[Proposition 4.2]{fu2010positiveSDE}.

\sk

If the diffusion coefficient does not vanish, we assume a non degenerate condition. This is used to deduce the pointwise exponent of the Brownian integral in \eqref{jd}, see Proposition \ref{difexp}.

\sk
\noindent(H3) Either $\inf_x |\sigma(x)|>\e$ for some $\e>0$ or $\sigma\equiv 0$.

\sk
  Finally we assume  

\sk
\noindent(H4)  $G$ is admissible in the sense of Definition \ref{defG}.

\begin{definition}
\label{defG}
%The set $\cG$ is the set of admissible functions   
A function $G : \rr\times \rr^* \to \rr$ is admissible if it satisfies:
\begin{itemize}
\item (Symmetry) For any $(x,y)\in\rr^2$ and non-zero $|z|\le 1$, $$ G(x,z) = \mathrm{sign}(z)\, |G(x,|z|)|  \ \mbox{ and }  \ G(x,z)G(y,z)>0.$$
\item  (Asymptotically stable-like) There exists a function $\be : \rr\to\rr$ with range in some compact set of $(0,2)$ such that for each $x\in\rr$ , $$\liminf_{z\rightarrow 0} \frac{\ln |G(x,z)|}{\ln |z|} = \frac{1}{ \beta(x)}.$$
 Furthermore, the following one-sided uniform bound holds : for any $\e>0$, there exists $r(\ep) >0 $ such that for any non-zero $|z|\le r(\ep)$ and $x\in\rr$,
$$|G(x,z)|\le |z|^{\frac{1}{\be(x)+\ep}}.$$
\item (Local Lipschitz condition) For each $m\in\nn^*$, there exists a finite constant $c_m$ such that for  $|x|, |y|\le m$  and non-zero $|z|\le 1$,
$$ \left| \frac{\ln |G(y,z)| -\ln |G(x,z)|}{\ln |z|} \right| \le c_m |x-y|.$$
%\item (Growth condition) For any $\e>0$, there exists finite $K_\e$ such that for each $x\in\rr$,
%$$\int_{\e<|z|\le 1} |G(x,z)|^2 \pi(\rd z) < K_\e(1+|x|^2).$$
\end{itemize}
\end{definition}

Let us comment (H4). The first item is artificial and used to simplify the statement of the results. If $G$ has different asymptotically stable-like behavior for $z>0$ and $z<0$, one can simply define $\be^+$ and $\be^-$ in the same manner as we define $\be$, and all the results hold with $\be$ replaced by $\be^+\vee\be^-$, see \cite[Chapter 3]{yang2016thesis} for more details on adding asymmetry to the SDE. The second item is clearly technical but general enough in the sense that it allows us to include the important class of stable-like processes with index function ranging in $(0,2)$ and non-degenerate stable-driven SDEs. In the presence of the second one, the third item is stronger than the usual local Lipschitz condition for the pathwise uniqueness of the solution, it is used to give an upper bound for the pointwise H\"older exponent, see the proof of Proposition \ref{jumpexp}.  The reader should keep in mind that when $G$ is admissible, one has intuitively 
$$ G(x,z)\  `` \sim " \ \mathrm{sign}(z)\, |z|^{1/\be(x)} g(x)$$
for some function $\be$ which ranges in $(0,2)$ and some non degenerate   function $g$ with some regularity.

%This class includes stable-like processes and the solutions to the Lévy stable-driven SDE given by 
%\begin{align*}
%M_t= \int_0^t\int_{C(0,1)} g(M_{s-}) dL^\al_s 
%\end{align*}
%with $L^\al$ an $\al$-stable Lévy process  and $g$ a reasonable function (see Section \ref{sec_comment} for details).

For the rest of the paper, we set
$$
t \in \R^+ \longmapsto \beta_M(t) = \beta(M(t)).
$$

The  quantity   $  \beta_M(t)$ is key: it  shall be understood as {\em the local Blumenthal-Getoor index of   $M$ at time $t$}, and governs the local behavior of $M$ at $t$.

%\begin{itemize}
%\sk\item They may have  infinite variations. We   exploit the martingale feature of the pure jump part of $M$, which turns out to be  the key element to obtain the lower bound for the pointwise regularity.
%
%\sk\item 
%They are Markov processes, and can be viewed as   generalization of L\'evy processes whose Blumenthal-Getoor index changes with time, depending on its  trajectory. This intuition is made precise in Section \ref{tang}, where we study the tangent processes to $M$.
%
%\sk\item The diffusion term and drift term are not as trivial (from a regularity standpoint) as in the case of L\'evy processes. For instance, the drift is not linear   and may be multifractal itself. \end{itemize}

\subsection{Main results}  
  
We state  now the multifractal  properties of $M$.  When the Brownian part does not vanish, the pointwise spectrum of $M$ takes a simple form, which is the main result of this paper.

 \begin{theorem}\label{mainwithdif}
Assume that (H1)-(H4) hold with non trivial $\sigma$. Then, almost surely, for each $t\in\rr^+$, the pointwise multifractal spectrum of $M$ at $t$ is
\begin{equation}
D_M(t,h) = \begin{cases} h \cdot \max(\beta_M(t),\beta_M(t-)) & \mbox{ if } h<1/2, \\ 1 & \mbox{ if } h=1/2, \\ -\infty & \mbox{ if } h>1/2. \end{cases} \nonumber
\end{equation}
In particular, if $t$ is a continuous time for $M$, the formula reduces to $D_M(t,h) = h \cdot \beta_M(t)$ when $h<1/2$.
\end{theorem}

From the pointwise  spectrum  of $M$ we deduce its local  spectrum, using \eqref{frompointtolocal}.

\begin{corollary}\label{corspec} Under the conditions of Theorem \ref{mainwithdif},  almost surely, for any open set $A\subset \R^+$,
%\begin{eqnarray*}
 % \gamma_{I } &  :=  &  \sup\Big \{ \beta_M(s) : \, s\in I \Big\}. \end{eqnarray*}
 the local multifractal spectrum of $M$ on $A$  is
\begin{equation}
D_M(A,h) =  \begin{cases}  h\cdot  \sup\Big \{ \beta_M(t) : \, t\in A \Big\} & \mbox{ if } h<1/2, \\ 1 & \mbox{ if } h=1/2, \\
-\infty & \mbox{ if } h>1/2. \end{cases} \nonumber
\end{equation}
 \end{corollary}
 
Observe that both pointwise and local spectrum are linear up to the exponent $h=1/2$. Recalling Jaffard's result (see \eqref{jaffard}), Corollary \ref{corspec} implies that the multifractal spectrum of $M$ looks like that of a Lévy process, except that the slope of the linear part of the spectrum is random and depends on the set on which  we compute the spectrum. This  remarkable property reflects the fact that the local Blumenthal-Getoor index of a jump diffusion $M$  depends on time. 
 
\mk 
 
Now we consider the case when the Brownian part vanishes. 

If the process is locally of bounded variation,   the compensated Poisson integral can be written as the difference of a non compensated Poisson integral and its compensator. Another drift $\wt G(x) = \int_{|z|\le 1} G(x,z)\,\pi(\rd z)$  appears and we need some regularity assumption on $b$ and this new drift:
 \begin{align*}
(\mathrm{H5}) \ \  \begin{cases}
\mbox{ either }    \ \ \  b\in C^2(\rr) \mbox{ and } \inf\{\be(x): x\in\rr\}\ge 1/2, \\
\ \ \mbox{ or }\ \ \ \ \  \sup\{k\in\nn: b \mbox{ and } \wt G \in C^{k}(\rr)\} \ge \sup\{1/\be(x), x\in\rr\}.
\end{cases}
\end{align*}
Let us comment (H5) before presenting the result. As one removes the Brownian part, there is a competition between the regularity of the drift (Lebesgue integral part of $M$), and that of the small jumps component (compensated Poisson integral part of $M$).  The point is that the drift is a functional of $M$ whose regularity is unknown {\it a priori}.  This is not a problem if one imposes more regularity on the drift coefficients.  It is quite similar  to some regularity assumptions on the Hurst function appearing in the study of multifractional Brownian motion. The literature devoted to the study of general mbm \cite{herbin2006, ayache2005multifractional, balanca2015mbm} clearly indicates that equivalent general result on jump diffusions would be much harder to obtain, and therefore out of the scope of the present work.

\begin{figure}
\begin{tikzpicture}[xscale=1.2,yscale=1.2]
{\small 
\draw [->] (0,-0.2) -- (0,3)  node [above] {$D_{M}(t,h)$};
\draw [->] (-0.4,0) -- (3,0) node [right] {$h$};
\draw [dotted] (0,-0.2) -- (0,-0.6);
\draw [fill] (0,-0.6)--(0, -1);
\draw [fill] (0,-0.8) node [left] {$-\infty$};
\draw [fill] (0,-0.8) circle [radius=0.03];
\draw [dotted] (0,1.5)--(1.5,1.5);
\draw [dotted] (1,1.5)--(1,0);
\draw [dotted] (1.5,1.5)--(1.5,0);
\draw [thick] (1,-0.8)--(2.9,-0.8);
\draw [thick, domain=0:1] plot({\x},{1*\x});
\draw [dotted, domain=1:1.5] plot({\x},{1*\x});
\draw [fill] (0,0) circle [radius=0.05];
\draw [fill] (1,1.5) circle [radius=0.05];
\draw [fill] (0,1.5) circle [radius=0.03];
\draw [fill] (1,0) circle [radius=0.03];
\draw [fill] (1,0) node [below] {$\frac{1}{2}$};
\draw [fill] (0,1.5) node [left] {$1$};
\draw [fill] (1.5,0) node [below] {$\frac{1}{\gamma}$};
\draw [->] (0.5, 2)--(0.5,0.9);
\draw [fill] (1,2) node [above] {slope $=\gamma$};
\draw [fill] (1.5,0) circle [radius=0.03];
}
\draw [dotted] (0,-0.8)--(1,-0.8);
\end{tikzpicture}
\begin{tikzpicture}[xscale=1.2, yscale=1.2]
{\small 
\draw [->] (0,-0.2) -- (0,3)  node [above] {$F_{{cont}}(c,\gamma,h)$};
\draw [->] (-0.4,0) -- (3,0) node [right] {$h$};
\draw [dotted] (0,-0.2) -- (0,-0.6);
\draw [fill] (0,-0.6)--(0, -1);
\draw [fill] (0,-0.8) node [left] {$-\infty$};
\draw [fill] (0,-0.8) circle [radius=0.03];
%\draw [thick, domain=0:1, color=black] plot({\x},{1.5*\x});
\draw [thick, domain=0:2] plot({\x},{0.75*\x});
%\draw [dotted, domain=0:1] plot({\x},{0.75*\x});
\draw [thick, domain=2:3] plot({\x}, {-0.8});
\draw [dotted, domain=0:2] plot({\x}, {-0.8}); 
%\draw [dotted] (1,0)--(1,1.5);
\draw [dotted] (2,0)--(2,1.5);
\draw [dotted] (0,1.5)--(2,1.5);
\draw [fill] (0,1.5) circle [radius=0.03];
\draw [fill] (0,1.5) node [left] {$1$}; 
\draw [fill] (0,0) circle [radius=0.05];
%\draw [fill] (1,0) circle [radius=0.03];
\draw [fill] (2,0) circle [radius=0.03];
\draw [fill] (2,0) node [below] {${1}/{\gamma}$};
%\draw [fill,color=red] (1,1.5) circle [radius=0.05];
%\draw [fill,color=green] (1,0.75) circle [radius=0.05];
\draw [fill,color=red] (2,1.5) circle [radius=0.05];
\draw [fill,color=blue] (2,0) circle [radius=0.05];
\draw [fill,color=green] (2,-0.8) circle [radius=0.05];
\draw [->] (1,2.5)--(1,1);
\draw [fill] (1.5,2.5) node [above] {slope $=\gamma$};
}
\end{tikzpicture}\begin{tikzpicture}[xscale=1.2, yscale=1.2]
{\small 
\draw [->] (0,-0.2) -- (0,3)  node [above] {$F_{\mbox{jump}}(c_1,c_2,\gamma_1,\gamma_2,h)$};
\draw [->] (-0.4,0) -- (3,0) node [right] {$h$};
\draw [dotted] (0,-0.2) -- (0,-0.6);
\draw [fill] (0,-0.6)--(0, -1);
\draw [fill] (0,-0.8) node [left] {$-\infty$};
\draw [fill] (0,-0.8) circle [radius=0.03];
\draw [thick, domain=0:1, color=black] plot({\x},{1.5*\x});
\draw [thick, domain=1:2, color=black] plot({\x},{0.75*\x});
\draw [dotted, domain=0:1] plot({\x},{0.75*\x});
\draw [thick, domain=2:3] plot({\x}, {-0.8});
\draw [dotted, domain=0:2] plot({\x}, {-0.8}); 
\draw [dotted] (1,0)--(1,1.5);
\draw [dotted] (2,0)--(2,1.5);
\draw [dotted] (0,1.5)--(2,1.5);
\draw [fill] (0,1.5) circle [radius=0.03];
\draw [fill] (0,1.5) node [left] {$1$};
\draw [fill] (0,0) circle [radius=0.05];
\draw [fill] (-0.1,-0.2) node [left] {$0$};
\draw [fill] (1,0) circle [radius=0.03];
\draw [fill] (1,0) node [below] {${1}/{\gamma_1}$};
\draw [fill] (2,0) circle [radius=0.05];
\draw [fill] (2,0) node [below] {${1}/{ \gamma_2}$};
\draw [fill,color=red] (1,1.5) circle [radius=0.05];
\draw [fill,color=green] (1,0.75) circle [radius=0.05];
\draw [fill,color=red] (2,1.5) circle [radius=0.05];
\draw [fill,color=blue] (2,0) circle [radius=0.05];
\draw [fill,color=green] (2,-0.8) circle [radius=0.05]; 
\draw [->] (0.5, 2.5)--(0.5, 1);
\draw [fill] (1,2.5) node [above] {slope $=\gamma_1 $};
\draw [->] (1.5, 2)--(1.5,1.2);
\draw [fill] (2,2) node [above] {slope $= \gamma_2$};
}
\end{tikzpicture}
\caption{Pointwise multifractal spectra of $M$.  Left: $\sigma\not= 0$. Center: $\sigma= 0$ and $t$ is a continuous time. Right: $\sigma= 0$ and $t$ is a jump time. See the statement of Theorem \ref{mainwithdif} and \ref{theo2} for the value of $\ga$, $\ga_1$ and $\ga_2$. The colored points correpond to possible values for the discontinuities of the pointwise spectra.} \label{figo2} 
\end{figure}
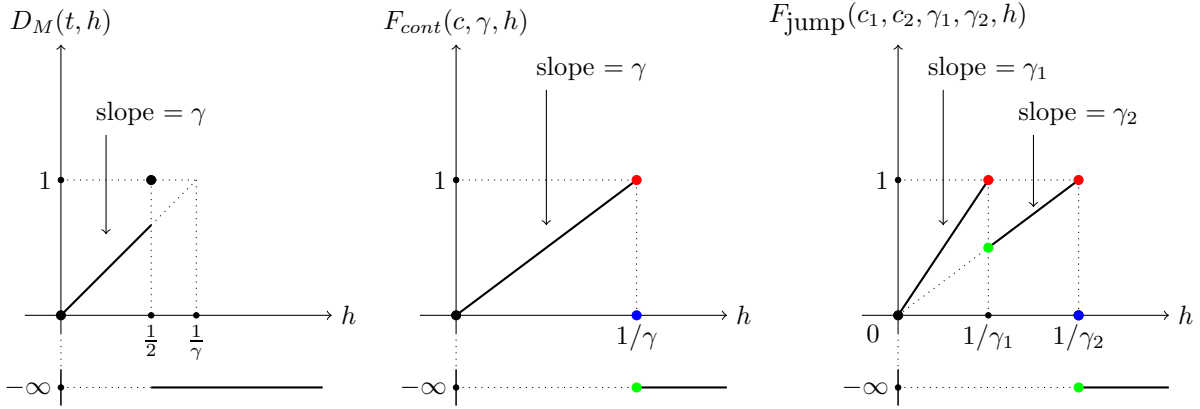

%%%%%%%%%%%%%%%%%%%%%%%%%%%%%%%%%%%%%%%%%%%%%%%%%%%%%%%%%%%%%%%%%%%%%%%%%%%%%%%%%%%
%%%%%%%%%%%%%%%%%%%%%%%%%%%%%%%%%%%%%%%%%%%%%%%%%%%%%%%%%%%%%%%%%%%%%%%%%%%%%%%%%%%
\begin{theorem}\label{theo2}
Assume that (H1)-(H5) hold with $\sigma\equiv 0$. Almost surely, 
 for each $t\ge 0$, the pointwise spectrum is
\begin{align*}
D_M(t,h)= \begin{cases}
h\cdot\max(\be_M(t), \be_M(t-)) & \mbox{ if }  0\le h<1/\max(\be_M(t), \be_M(t-)), \\ 
h\cdot\min(\be_M(t), \be_M(t-)) & \mbox{ if }  1/\max(\be_M(t), \be_M(t-)) < h < 1/\min(\be_M(t), \be_M(t-)), \\
-\infty  & \mbox{ if }  h> 1/\min(\be_M(t), \be_M(t-)).
\end{cases}
\end{align*}
\end{theorem}

The idea is that the pointwise spectrum of $M$ is determined by the local index process $t\mapsto\be_M(t)$. When $t$ is a continuous time for $M$ (so it is for $\be_M$ as $\be$ is continuous by (H4)), the local index does not vary much around $\be_M(t)$,  the resulting spectrum is a linear function with slope $\be_M(t)$ much as in Jaffard's Theorem; when $t$ is a jump time, one has two characteristic index $\be_M(t-)$ and $\be_M(t)$ around $t$, restricting ourselves to $(t-\de, t)$ (resp. $(t,t+\de)$) for $\de>0$ results in a linear function with slope $\be_M(t-)$ (resp. $\be_M(t)$), combining them results in the superposition of two linear functions.  As $\be_M$ ranges in $(0,2)$, only the steeper one of these two linear functions is seen if one adds a point at $(1/2, 1)\in\rr^2$, which corresponds to the $\sigma$ non trivial case.

Note that none of those extreme values of $h$ - discontinuities of the pointwise spectrum - are discussed in this theorem. They 
 will be entirely treated in Theorems   \ref{mainresult} and \ref{mainresult2} of Section \ref{secspec}. It is more complicated to state, since many cases must be distinguished according to various relationships between $t$, $M_t$ and $\be$. In particular, the pointwise spectrum at those particular $h$ might be the right-continuous or left-continuous extension (or neither) of the one obtained in Theorem \ref{theo2}.

The local spectrum can thus be deduced via \eqref{frompointtolocal}, and it is a corollary of  Theorem  \ref{mainresult} and \ref{mainresult2}.

\begin{corollary}\label{corspec2} Assume the conditions of Theorem \ref{theo2}. Let $J$ be the set of jump time of $M$, write $\be_M(J)^{-1}=\{1/\be_M(t): t\in J\}$.  Let $I$ be any open set in $\rr^+$  and
\begin{eqnarray*}
  \gamma_{I}(h) &  :=  &  \sup\Big \{ \beta_M(s) : \, s\in I, \  \beta_M(s) \leq 1/h \Big\} ,\\
  \widetilde\gamma_{I} &  :=  &  \inf\Big \{ \beta_M(s) : \, s\in I \Big\} .
\end{eqnarray*}
With probability one,  the local multifractal spectrum of $M$ on $I$ is
 
\begin{equation}
D_M(I,h) =  \begin{cases} h\cdot   \gamma_{I}(h) 
& \mbox{ if } h<1/ \widetilde\gamma_{I} \ \mbox{ and } h \notin (\beta(J))^{-1}, \\
-\infty & \mbox{ if } h>1/ \widetilde \gamma_I. \end{cases} \nonumber
\end{equation}
 
\end{corollary}

Theorems combined with their corollaries are compared in Figure \ref{figo2} and \ref{fig1}.

The difference between the corollaries follows from the fact that the continuous component of $M$ has regularity 1/2 at every point, so the complicated part of the multifractal spectrum ($h>1/2$) in Corollary \ref{corspec2} disappears (see Figure \ref{fig1}).

Observe that we do not give the value of the spectrum on the countable set $ \big(\beta_M ({J})\big)^{-1}$. This is due to the occurrence of various delicate situations depending on the trajectory of $M$, which  are described in Section \ref{secspec}.

\begin{figure}
\begin{tikzpicture}[xscale=1.3,yscale=1.3]
{\small 
\draw [->] (0,-0.2) -- (0,3)  node [above] {$D_{M}(I,h)$};
\draw [->] (-0.4,0) -- (3,0) node [right] {$h$};
\draw [dotted] (0,-0.2) -- (0,-0.6);
\draw [fill] (0,-0.6)--(0, -1);
\draw [fill] (0,-0.8) node [left] {$-\infty$};
\draw [fill] (0,-0.8) circle [radius=0.03];
\draw [dotted] (0,1.5)--(1.5,1.5);
\draw [dotted] (1,1.5)--(1,0);
\draw [dotted] (1.5,1.5)--(1.5,0);
\draw [thick] (1,-0.8)--(2.9,-0.8);
\draw [thick, domain=0:1] plot({\x},{1*\x});
\draw [dotted, domain=1:1.5] plot({\x},{1*\x});
\draw [fill] (0,0) circle [radius=0.05];
\draw [fill] (1,1.5) circle [radius=0.05];
\draw [fill] (0,1.5) circle [radius=0.03];
\draw [fill] (1,0) circle [radius=0.03];
\draw [fill] (1,0) node [below] {$\frac{1}{2}$};
\draw [fill] (0,1.5) node [left] {$1$};
\draw [fill] (1.5,0) node [below] {$\frac{1}{\gamma_I}$};
\draw [->] (0.5, 2)--(0.5,0.9);
\draw [fill] (1,2) node [above] {slope $=\gamma_I$};
\draw [fill] (1.5,0) circle [radius=0.03];
}
\draw [dotted] (0,-0.8)--(1,-0.8);
\end{tikzpicture}
\begin{tikzpicture}[xscale=1.3, yscale=1.3]
{\small 
\draw [->] (0,-0.2) -- (0,3)  node [above] {$D_{M}(I,h)$};
\draw [->] (-0.4,0) -- (3,0) node [right] {$h$};
\draw [dotted] (0,-0.2) -- (0,-0.6);
\draw [fill] (0,-0.6)--(0, -1);
\draw [fill] (0,-0.8) node [left] {$-\infty$};
\draw [fill] (0,-0.8) circle [radius=0.03];
\draw [thick, domain=0:1.5] plot({\x},{1*\x});
\draw [dotted, domain=0:1.5] plot({\x},{(15/17)*\x});
\draw [thick, domain=1.5:1.7] plot({\x}, {(15/17)*\x});
\draw [dotted, domain=0:1.7] plot({\x},{(15/19)*\x});
\draw [thick, domain=1.7:1.9] plot({\x}, {(15/19)*\x});
\draw [dotted, domain=0:1.9] plot({\x},{(15/20)*\x});
\draw [thick, domain=1.9:2.0] plot({\x}, {(15/20)*\x});
\draw [dotted, domain=0:2] plot({\x},{(15/22)*\x});
\draw [thick, domain=2:2.2] plot({\x}, {(15/22)*\x});
\draw [dotted, domain=0:2.2] plot({\x},{(15/23)*\x});
\draw [thick, domain=2.2:2.3] plot({\x}, {(15/23)*\x});
\draw [dotted, domain=0:2.3] plot({\x},{(15/24)*\x});
\draw [thick, domain=2.3:2.4] plot({\x}, {(15/24)*\x});
\draw [dotted, domain=0:2.4] plot({\x},{(15/25)*\x});
\draw [thick, domain=2.4:2.5] plot({\x}, {(15/25)*\x});
\draw [dotted, domain=0:2.5] plot({\x}, {-0.8}); 
\draw [dotted] (2.5,0)--(2.5,1.5);
\draw [dotted] (0,1.5)--(2.5,1.5);
\draw [thick, domain=2.5:2.9] plot({\x},{-0.8}); 
\draw [fill] (0,1.5) circle [radius=0.03];
\draw [fill] (0,1.5) node [left] {$1$}; 
\draw [fill] (0,0) circle [radius=0.05];
\draw [fill] (1.5,0) circle [radius=0.03];
\draw [fill] (2.5,0) circle [radius=0.03];
\draw [fill] (2.5,0) node [below] {$\frac{1}{\widetilde \gamma_I}$};
\draw [fill] (1.5,0) node [below] {$\frac{1}{\gamma_I}$};
\draw [dotted] (1.5,1.5)--(1.5,0);
\draw [fill] (1,0) circle [radius=0.03];
\draw [fill] (1,0) node [below] {$\frac{1}{2}$};
\draw [->] (0.5, 2)--(0.5,0.9);
\draw [fill] (1,2) node [above] {slope $=\gamma_I$};
\draw [dotted] (1,1.5)--(1,0);
}\end{tikzpicture}
\caption{Local multifractal spectrum of $M$ in the interval $I$ when $\sigma\not\equiv 0$ (left) and $\sigma\equiv 0$ (right). The right figure is a representation, since there is a countable number of small affine parts. When $\sigma\not\equiv 0$, the Brownian integral "hides" the  complicated right part of $D_M(I,.)$.}\label{fig1}
\end{figure}
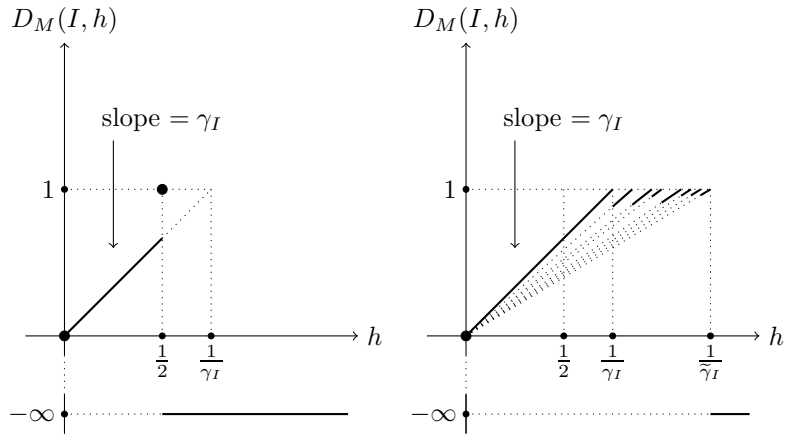

\sk

\subsection{Extensions}\label{sec:extension}
This work is a first step of the long range research project of understanding path regularity of Markov processes. Considered in \cite[Chapter 3]{yang2016thesis} are  multidimensional versions of  similar SDEs with anisotropic $G$ and more general intensity measures $\pi$ under the condition that the associated Poisson point process satisfies some good covering properties.  Other dimensional properties of stochastic processes, such as dimensions of the range, of the graph of $M$, are important mathematical properties with application in physics for modeling purposes, and are investigated in \cite{yang2015range}.  

 Certain classes of Markov processes having a SDE representation are not covered by our main theorems due to the presence of the degenerate coefficients and variable jump rate. This is the case of continuous state branching processes \cite{fu2010positiveSDE} and of positive self-similar Markov processes \cite{doring2012pssmp}. It would be very interesting to determine their multifractal structure. 

In terms of application, a recent original method in model selection of signal processing consists in estimating the parameters of the multifractal spectrum of the model, see \cite{abry2015irregularities}.  Due to the importance of jump diffusion model in physics and finance \cite{chudley1961, amin1993},  it would be of much interest to develop statistical tools to estimate multifractal parameters for SDE with jumps.  

\subsection{Plan of the paper}
In Section 2, first properties of the process $M$ are given. In Section 3, we prove some technical estimates using a new slicing technique.  As a crucial step to derive a multifractal spectrum for $M$,  we state - in Section 4 - Theorem \ref{exponent} on the pointwise H\"older exponent of $M$, whose proof is given in Section 5.   In Section 6, we first compute the pointwise spectrum - Theorem \ref{mainwithdif} - when $\sigma$ is non trivial, and  the linear parts of the pointwise spectrum - Theorem \ref{theo2} - when the Brownian integral vanishes. Then we complete the study by stating  and proving  Theorem \ref{mainresult} and Theorem \ref{mainresult2} which treat the discontinuities of the pointwise spectrum when $\sigma$ is trivial.  Finally, as an application of our main results, we discuss variable order stable-like processes and non-degenerate stable-driven SDEs in Section 7.  Auxiliary results are given  in Appendices, which contain also a discussion on the existence of tangent processes.  Throughout, $C$ denotes a generic finite positive constant whose value may change in each appearance.

\sk

%%%%%%%%%%%%%%%%%%%%%%%%%%%%%%%%%%%%%%%%%%%%%%%%%%%%%%%%%%%%%%%%%%%%%%
%%%%%%%%%%%%%%%%%%%%%%%%%%%%%%%%%%%%%%%%%%%%%%%%%%%%%%%%%%%%%%%%%%%%%%
\section{Basic properties of $M$}\label{secadmissible}%%%%%%%%%%%%%%%%%%%%%%%%%%%%%%%%%%%%%%%%%%%
%%%%%%%%%%%%%%%%%%%%%%%%%%%%%%%%%%%%%%%%%%%%%%%%%%%%%%%%%%%%%%%%%%%%%%
%%%%%%%%%%%%%%%%%%%%%%%%%%%%%%%%%%%%%%%%%%%%%%%%%%%%%%%%%%%%%%%%%%%%%%
Throughout the section, we assume (H1)-(H4).

%%%%%%%%%%%%%%%%%%%%%%%%%%%%%%%%%%%%%%%%%%%%%
%%%%%%%%%%%%%%%%%%%%%%%%%%%%%%%%%%%%%%%%%%%%%
%%%%%%%%%%%%%%%%%%%%%%%%%%%%%%%%%%%%%%%%%%%%%
%%%%%%%%%%%%%%%%%%%%%%%%%%%%%%%%%%%%%%%%%%%%%
%%%%%%%%%%%%%%%%%%%%%%%%%%%%%%%%%%%%%%%%%%%%%
%%%%%%%%%%%%%%%%%%%%%%%%%%%%%%%%%%%%%%%%%%%%%
%%%%%%%%%%%%%%%%%%%%%%%%%%%%%%%%%%%%%%%%%%%%%

%%%%%%%%%%%%%%%%%%%%%%%%%%%%%%%%%%%%%%%%%%%%%%%%%%%%%%%%%%%%%%%%%%%%%
\begin{proposition}\label{strongsolution}
  The SDE \eqref{jd} has a unique c\`adl\`ag strong solution which is a $(\cF_t)$ strong Markov process. 
\end{proposition}

By an interlacement procedure for the non compensated Poisson integral (see for instance the proof of \cite[Proposition 2.4]{fu2010positiveSDE}), it is enough to consider the SDE \eqref{jd} without non compensated Poisson integral, i.e. $F=0$.  Usual Picard iteration, Gronwall Lemma and localization procedure entail the existence of a unique strong solution for the modified SDE once we check the usual linear growth condition and local Lipschitz condition for the coefficients. By (H1)-(H2), it remains to prove the local Lipschitz condition for $G$, that is,  for each $m\in\nn^*$ there is a finite constant $c_m$ such that for $|x|,|y|\le m$, 
\begin{equation} 
\int_{|z|\le 1} |G(x,z)-G(y,z)|^2 \,\pi(\rd z) \le c_m|x-y|^2. \nonumber 
\end{equation}
  This is checked in Appendix A.  The strong Markov property follows from pathwise uniqueness.
%%%%%%%%%%%%%%%%%%%%%%%%%%%%%%%%%%%%%%%%%%%%%%%%%%%%%%%%%%%%%%%%%%%%%
 
%%%%%%%%%%%%%%%%%%%%%%%%%%%%%%%%%%%%%%%%%%%%%%%%%%%%%%%%%%%%%%%%
%\begin{proposition}\label{propcvuni}
%Let $M$ be the strong solution of \eqref{jd}. Then the compensated Poisson integral 
%$$X : t\mapsto \int_0^t\int_{|z|\le 1} G(M_{s-},z)\, \wt N (\rd s, \rd z)$$
% is a $L^2$-martingale with respect to $(\cF_t)$ i.e. it is a $(\cF_t)$ martingale and for each $t\ge 0$, $\E[|X_t|^2]<\infty$.
%%  and there is a strictly positive sequence $(\al_n)_{n\geq 1}$ with $\al_n \rightarrow_{n\to +\infty} 0$ such that a.s., 
%%\begin{equation}
%%\calZ_t =  \lim_{n\rightarrow +\infty} \intot \int_{C({\al_n},1)} G(M_{s-},z) \,\wn(\rd s, \rd z), \mbox{ uniformly in } [0,1]. \nonumber
%%\end{equation}
%\end{proposition}
%\begin{proof}
%By a classical argument using uniform convergence on every compact set in probability of a sequence of $L^2$-martingale, it is enough to show the following (see \cite[Theorem 4.2.3]{applebaum2009})
%\begin{align}\label{l2mart}
%\forall \,t>0, \ \ \intot \int_{|z|\le 1} \E [G(M_{s-},z)^2]\,\rd s \, \pi(\rd z) < +\infty .
%\end{align}
%Thanks to (H1), the previous term is bounded above by 
%$$K\int_0^t (1+\E[M_s^2]) \,\rd s$$
%which is bounded by {\it a priori} estimates under linear growth and local Lipschitz condition in \cite[Lemma 114]{situ2005} which says that the solution $M$ is bounded in $L^2(\Omega)$ up to any time $t\ge 0$ as soon as $M_0\in L^2(\Omega)$.
%\end{proof}

\begin{proposition}\label{generator}
The generator of the Markov process $M$ is 
\begin{multline*}
\cL f(x) = b(x)f'(x) + \frac{1}{2}\sigma^2(x)f''(x) 
+ \int_{|z|\le 1} \[ f(x+G(x,z))-f(x)-G(x,z)f'(x)\] \,\pi(\rd z) \\ 
+ \int_{|z|>1} \[ f(x+F(x,z))-f(x)\] \,\pi(\rd z)
\end{multline*}
for any $f\in C^2_c(\rr)$, space of twice continuously differentiable functions with compact support.
\end{proposition}
\begin{proof}
By It\^o's formula for jump processes (see \cite[page 57]{jacod2003}), one has for each $f\in C^2_c(\rr)$ and any initial distribution $\mu$, 
\begin{multline*}
f(M_t) -f(M_0) - \int_0^t \cL f(M_s)\,\rd s \\
= \int_0^t f'(M_{s-})\sigma(M_{\sm}) \,\rd B_s 
  + \int_0^t\int_{|z|\le 1}[f(M_{\sm}+G(M_{\sm},z)) - f(M_{\sm})]\,\wn(\rd s,\rd z) \\
  + \int_0^t\int_{|z|> 1}[f(M_{\sm}+F(M_{\sm},z)) - f(M_{\sm})]\,\wn(\rd s,\rd z).
\end{multline*}
The integrand of the Brownian integral and that of the compensated Poisson integrals are both bounded since $f\in C^2_c(\rr)$ so that the right-hand side of the above equality is a martingale.  As pathwise uniqueness for \eqref{jd} implies uniqueness in law (see \cite[Theorem 1.1]{barczy2015}) which is equivalent to uniqueness of the martingale problem associated with $\cL$ with initial distribution $\mu$, see \cite[Corollary 2.5]{kurtz2011},  we have proved that the generator of $M$ on $C^2_c(\rr)$ is indeed $\cL$. It is easy  to check the condition (14) in \cite{kurtz2011} in order to apply  Corollary 2.5 therein.  We omit the details.
\end{proof}

\sk

{\em  Hereafter, we restrict our study to the time interval $[0,1]$, the extension to $\rr^+$ is straightforward. }

%%%%%%%%%%%%%%%%%%%%%%%%%%%%%%%%%%%%%%%%%%%%%%%%%%%%%
%%%%%%%%%%%%%%%%%%%%%%%%%%%%%%%%%%%%%%%%%%%%%%%%%%%%%
%%%%%%%%%%%%%%%%%%%%%%%%%%%%%%%%%%%%%%%%%%%%%%%%%%%%%
%%%%%%%%%%%%%%%%%%%%%%%%%%%%%%%%%%%%%%%%%%%%%%%%%%%%%
%%%%%%%%%%%%%%%%%%%%%%%%%%%%%%%%%%%%%%%%%%%%%%%%%%%%%
%%%%%%%%%%%%%%%%%%%%%%%%%%%%%%%%%%%%%%%%%%%%%%%%%%%%%
%%%%%%%%%%%%%%%%%%%%%%%%%%%%%%%%%%%%%%%%%%%%%%%%%%%%%
\section{Technical estimates}\label{upj}

Let us provide a uniform in $[0,1]$ increment estimate for the compensated Poisson integral truncated the large jumps for a family of different truncations.  Balança \cite{balanca2014levy}  has proved a similar result for L\'evy processes. To overcome the difficulty that our process does not have stationary increments, we develop a slicing technique which consists in cutting the compensated Poisson integral according to the value of the local index process $t\mapsto \be_M(t)$, somehow \`a la Lebesgue.  Each sliced process has a more or less constant local index for which we are able to obtain precise estimates with approximately right order.  Adding up these slices gives the desired estimates.

We need some notations. Set for any interval $I=[a,b]\subset[0,1]$
\begin{align*}
\be_{M}^{I,n}  = \( \sup_{u\in I} \beta_M(u)+\frac{2}{n} \)  \ \   \mbox{ and } \ \  
\wh{\be}_{M}^{I,n} = \( \sup_{u\in I\pm 2^{-n}}\beta_M(u)+\frac{2}{n}\) \ ,
 \end{align*}
 where $I\pm 2^{-n} = [a-2^{-n}, b+2^{-n}]$.

%where
%\begin{equation*}
%[s,t] = \begin{cases}
%[ s,t ] & \mbox{if } s<t, \\
%[t,s]  & \mbox{otherwise, }
%\end{cases}  \ \ \ \  [s,t]\pm 2^{-n} = \begin{cases} 
%[ s-2^{-n},t+2^{-n} ]  & \mbox{if } s<t, \\
%  [t-2^{-n},s+2^{-n}]  & \mbox{otherwise. }
%\end{cases}
%\end{equation*}
%%%%%%%%%%%%%%%%%%%%%%%%%%%%%%%%%%%%%%%%%%%%%%%%%%%%%%%%%%%%%%%%%
%%%%%%%%%%%%%%%%%%%%%%%%%%%%%%%%%%%%%%%%%%%%%%%%%%%%%%%%%%%%%%%%%%
\begin{proposition}\label{LElemme}
There exists finite positive constants $K$, $\e_0$ such that for each $\de>1$,  $0<\ep<\e_0$, $n\ge n_0$ (depending only on $\ep$ and $G$)
\begin{align*}
\pp \( \sup_{\overset{|s-t|\le 2^{-n}}{s<t\in\zu}} \lba 2^{\frac{n}{\de \left (\widehat{\be}_{M}^{[s,t],n}+\ep \right)}}  \int_s^t\int_{|z|\le 2^{-n/\de}}  \hspace{-2mm} G(M_{\um},z) \,\wn (\rd u,\rd z) \rba \ge 8 n^2 \) \le K e^{-n}.
\end{align*} 
\end{proposition}

In words, if we look at the increment in a small interval $I$ of the compensated Poisson integral truncated the jumps of size larger than $|I|^{1/\de}$ of the underlying Poisson point process,  we observe with high probability that the mentioned increment is bounded above by $|I|^{1/(\de\wh\be^{I,n}_M)}$ with some logarithmic correction. It is remarkable that this statement holds uniformly for all small $I\subset[0,1]$.

%%%%%%%%%%%%%%%%%%%%%%%%%%%%%%%%%%%%%%%%%%%%%%%%%%%%%%%%%%%%%%%%%%
%%%%%%%%%%%%%%%%%%%%%%%%%%%%%%%%%%%%%%%%%%%%%%%%%%%%%%%%%%%%%%%%%%

% Here, the idea is to exploit the martingale nature of our process and some ``freezing" procedure for the local upper index process $t\mapsto\beta_M(t)$. Intuitively, in the neighborhood of a  continuity  point $t$,  since  $\beta_M(t)$ is also continuous at $t$, one may say that our process behaves locally like a Lévy process with Blumenthal-Getoor's   index $\beta_M(t)$.  A good way to make explicit this intuition is to cut the index process in the spirit of Lebesgue integral. Roughly speaking, we decompose the first term of \eqref{decomposition} as a sum of $m$ processes $\calP^j$,  whose  local index   takes value in $[2j/m,2(j+1)/m)$. When $m$ becomes large, the local behavior of these processes is comparable with that of some Lévy process, in probability.

%Proposition  \ref {LElemme} brings information  about the uniform increment estimate of $\calZ$. In contrast with the L\'evy case, it is remarkable that the exponent   depends on two parameters:   the approximation rate  $\de_t$ and   $\beta_M(t)$, both  random and correlated with  $M$. This observation complicates in the proof.   
%\sk

%%%%%%%%%%%%%%%%%%%%%%%%%%%%%%%%%%%%%%%%%%%%%%%%%%%%%%%%%%%%%%%
The proof is decomposed into several lemmas. The first gives an increment estimate in any dyadic interval with ``either constant ($\approx 2k/n$) or zero" index. 
 Set $I_{n,\ell}= [t_{n,\ell},t_{n,\ell+1}]$ and $t_{n,\ell}= \ell 2^{-n}$.

%%%%%%%%%%%%%%%%%%%%%%%%%%%%%%%%%%%%%%%%%%%%%%%%%%%
%% First dyadic interval, freezed index %%%%%%%%%%%
%%%%%%%%%%%%%%%%%%%%%%%%%%%%%%%%%%%%%%%%%%%%%%%%%%%
\begin{lemma}\label{aux1} There exist finite positive constants $K$, $\e_0$ such that for each $\de>1$, $0<\ep<\e_0$, $n\ge n_0$ (depending only on $\e$ and $G$),  and  $\ell\in\{0, \cdots, 2^n-1 \}$, $k \in \{ 0,\cdots, n-1 \}$,
\begin{align}
&\ \ \pp \( \sup_{t\le 2^{-n}} \lba \int_{t_{n,\ell}}^{t_{n,\ell}+t} \int_{|z|\le{2^{-n/{\delta}}}} G(M_{\sm},z) \indiq_{\beta_M(s-)\in \left[\frac{2k}{n},\frac{2k+2}{n}\right)} \,\wn(\rd s,\rd z) \rba \geq 2n 2^{-\frac{n}{\delta(2k+2+n\ep)/n}} \) \le K e^{-2n}.\nonumber
\end{align}
\end{lemma}
Observe that the estimate is the same for any dyadic interval, this is because the underlying process has an almost constant index, which mimics the stationarity of increments of L\'evy processes.
%%%%%%%%%%%%%%%%%%%%%%%%%%%%%%%%%%%%%%%%%%%%%%%%%%%%%%
\begin{proof}
Set 
\begin{align*}
H_k(s,z)&:= 2^{\frac{n}{\de(2k+2+n\ep)/n}} G(M_{s-},z) \indiq_{\beta_M(s-)\in \left[\frac{2k}{n},\frac{2k+2}{n}\right)} \indiq_{|z|\le 2^{-n/\de}}, \\
P_k(t) &:= \int_{\ell 2^{-n}}^{\ell 2^{-n}+t} \int_{|z|\le 1} H_k(s,z)\,\wn (\rd s,\rd z). 
\end{align*}  
These processes depend clearly on $\de$, $\e$, $n$ and $\ell$. We ignore them for notational simplicity.  First we check that $P_k$ is a $L^2$ martingale.   It suffices to show  the following (see \cite[Theorem 4.2.3]{applebaum2009})
\begin{align}\label{l2mart}
\forall \,t>0, \ \ \int_{\ell 2^{-n}}^{\ell 2^{-n}+t} \int_{|z|\le 1} \E [H_k(s,z)^2]\,\rd s \, \pi(\rd z) < +\infty .
\end{align}
By the asymptotically stable-like assumption (H4),  for each  $t\ge 0$, one has
\begin{align*}
\int_{t_{n,\ell}}^{t_{n,\ell} +t}\!\!\! \int_{|z|\le 1} \!\!\!\!\E[H_k(s,z)^2] \,\pi(\rd z)\rd s
=\int_{t_{n,\ell}}^{t_{n,\ell} +t} \!\!\!\! 2^{\frac{2n}{\de(2k+2+n\ep)/n}}  \E \[ \indiq_{\beta_M(s-)\in \left[\frac{2k}{n},\frac{2k+2}{n}\right)} \int_{|z|\le 2^{-n/\de}} \!\!\!\!\!\!\!\! G(M_{\sm},z)^2  \,\frac{\rd z}{z^2} \] \rd s 
\end{align*}
which is bounded above by
\begin{align*}
\int_{t_{n,\ell}}^{t_{n,\ell} +t} 2^{\frac{2n}{\de(2k+2+n\ep)/n}}  \E \[ \indiq_{\beta_M(s-)\in \left[\frac{2k}{n},\frac{2k+2}{n}\right)} \int_{|z|\le 2^{-\frac{n}{\de}}}  |z|^{\frac{2}{\beta_M(s-)+\ep/2}-2} \,\rd z \] \rd s 
\end{align*}
for $n\ge n_0$ where $n_0$ depends only on $G$ and $\e$.  Let $\e_0 =\mathrm{dist}(\ov{\be(\rr)}, 2)$. Simple calculus implies that last integral is  bounded above by $C t 2^{n/\de}$ with  $C =(2/(2-\e_0/2) -1)^{-1}$.
Thus one has for all $t\le 2^{-n}$,
\begin{align}\label{const}
\int_{t_{n,\ell}}^{t_{n,\ell} +t} \int_{|z|\le 1} \E [H_k(s,z)^2] \,\pi(\rd z) \rd s \le C2^{-n}2^{n/\de} \le   C, 
\end{align}
which proves that  $(t\mapsto P_k(t))_{t\le 2^{-n}}$ is a $L^2$ martingale.

%\begin{align*}
%&\le C \int_0^t \E \[ 2^{\frac{2m}{\de(2j+2+m\ep)/m}} \indiq_{\beta_M(s)m\in \left[\frac{2j}{m},\frac{2j+2}{m}\right)}  2^{-\frac{m}{\de}(\frac{2}{\beta_M(s)m+\ep/2}-1)} \] ds \ala
%&= C \int_0^t \E \[ 2^{\frac{m}{\de} \( \frac{2}{(2j+2)/m+\ep} - \frac{2}{\beta_M(s)m+\ep/2} + 1 \) } \indiq_{\beta_M(s)m\in \left[\frac{2j}{m},\frac{2j+2}{m}\right)} \] ds,\nonumber
%\end{align*}
% Hence, 

%{\bf D'après $\cH_1$, pour chaque $x$ fixé, on peut trouver $z(x,\ep)$ tel que pour tout $z\in A(0,z(x,\ep))$, $G(x,z)\le z^{1/(\be(x)+\ep)}$. Mais dans la formule, on intègre [ds], il y a un continuum de valeurs de $x$. Il est tout à fait possible que $\inf_{x} z(x,\ep)=0$, on ne peut avoir l'inégalité désirée, peu importe la valeur de $m$.  
%
%Je crois que une condition uniforme en $x$ soit nécessaire pour la raison suivante. On fixe $x$, comme avant, on peut trouver $z(x,\ep)$ tel que pour tout $z\in A(0,z(x,\ep))$, $G(x,z)\le z^{1/(\be(x)+\ep)}$. Ensuite, par $\cH_2$, l'application $x\mapsto \ln G(x,z) / \ln z$ est continue, donc pour $y$ proche de $x$ (disons $|x-y|\le \eta$), on a $G(y,z)\le z^{1/(\be(x)+ 2\ep)}$. Mais pour le variable $G(M_{\sm},z)$,  l'accroissment de $M_{\sm}$ est aléatoire, à priori, on ne peut trouver un $m$ déterministe tel que $\forall s, t \in [0,2^{-m/\de}]$, $|M_s -M_t|\le \eta$ 
%Donc on a besoin une condition uniforme en $x$, pour que l'inégalité marche dès que $m$ soit suffisament grand. }

One deduces by convexity and Jensen's inequality that $t\mapsto e^{P_k(t)}$ and $t\mapsto e^{-P_k(t)}$ are submartingales. By Doob's $L^1$ maximal inequality for positive submartingales,
\begin{align*}
 \pp \( \sup_{t\le 2^{-n}} |P_k(t)| \ge 2n \) &\le  \pp \( \sup_{t\le 2^{-n}} e^{P_k(t)} \ge e^{2n} \) + \pp \( \sup_{t\le 2^{-n}} e^{-P_k(t)} \ge e^{2n} \) \\
&\le e^{-2n} \( \E \[ e^{P_k(2^{-n})} \] + \E \[ e^{-P_k(2^{-n})} \] \) . \nonumber
\end{align*}
Now we show that $\E[ e^{P_k(2^{-n})}]$  and  $\E[ e^{-P_k(2^{-n})}]$ are finite and independent of $n$, which completes the proof. It suffices to study $e^{P_k(t)}$. Applying Itô's Formula for jump processes,
\begin{multline}
e^{P_k(t)} = 1 + \int_{t_{n,\ell}}^{t_{n,\ell} +t} \int_{|z|\le 1
} e^{P_k(s-)}\( e^{H_k(s,z)}-1 \) \,\wn(\rd s, \rd z) \\
 + \int_{t_{n,\ell}}^{t_{n,\ell} +t} \int_{|z|\le 1
} e^{P_k({\sm})} \( e^{H_k(s,z)}-1 -H_k(s,z) \) \frac{\rd z}{z^2}\rd s. \label{eqIto}
\end{multline}
For all $r>0$, set $\tau_r = \inf\{t\ge 0 : |P_k(t)|\ge r \}$. Observe that for $s\ge 0$ and $n\ge n_0$,
\begin{align*}
|H_k(s,z)| \le2^{\frac{n}{\de(2k+2+n\ep)/n}} |z|^{\frac{1}{\beta_M(s-)+\ep/2}} \indiq_{|z|\le 2^{-n/\de}} \indiq_{\beta_M(s-)\in \left[\frac{2k}{n},\frac{2k+2}{n}\right)}   \leq  1.
\end{align*}
Thus, using $|e^u-1-u|\le |u|^2$ for $|u|\le 1$ and taking expectation in  \eqref{eqIto} yields that for $t\ge 0$,
\begin{align*}
\E [e^{P_k(t\wedge \tau_r)}] &= 1 + \E \[ \int_{t_{n,\ell}}^{t_{n,\ell} +t\wedge \tau_r} \int_{|z|\le 1
} e^{P_k(s-)} \( e^{H_k(s,z)}-1 -H_k(s,z) \) \,\frac{dz}{z^2}\,\rd s \] \\
&\le 1 + \E \[ \int_{t_{n,\ell}}^{t_{n,\ell} +t\wedge \tau_r}\int_{|z|\le 1
} e^{P_k(\sm)} H_k(s,z)^2 \,\frac{\rd z}{z^2}\,\rd s \] \\
&= 1+ \E \[ \int_{t_{n,\ell}}^{t_{n,\ell} +t\wedge\tau_r} \int_{|z|\le {2^{-n/\de}}} e^{P_k(\sm)} 2^{\frac{2n}{\de(2k+2+n\ep)/n}} \indiq_{\beta_M(\sm)\in \left[\frac{2k}{n},\frac{2k+2}{n}\right)} G(M_{\sm},z)^2 \,\frac{\rd z}{z^2}\,\rd s \] 
\end{align*}
%\begin{align*}
%&\le 1+ 2 \E \[ \int_0^t e^{\calP^j_{\sm}} 2^{\frac{2m}{\de(2j+2+m\ep)/m}} \indiq_{\beta_M(s)m\in \left[\frac{2j}{m},\frac{2j+2}{m}\right)} \int_0^{2^{-m/\de}}  z^{\frac{2}{\beta_M(s)m+\ep/2}-2} \,dz\,ds \]  ,
%\end{align*}
Using (H4) again and calculus shows that the integral inside the expectation is bounded above by 
\begin{align*}
C\int_{t_{n,\ell}}^{t_{n,\ell} +t\wedge\tau_r} e^{P_k(\sm)} 2^{n/\de} \,\rd s
\end{align*}
where $C$ is obtained in \eqref{const}.  Hence, 
\begin{align*}
\E [e^{P_k(t\wedge\tau_r)}] \le 1+ C  \int_{t_{n,\ell}}^{t_{n,\ell} +t} \E [e^{P_k(s\wedge\tau_r)}] 2^{n/\de} \,\rd s.
\end{align*}
% We deduce that 
%\begin{align*}
%\E [e^{\calP^j_t}] &\le 1+ C \E \[ \int_0^t e^{\calP^j_{\sm}} 2^{\frac{2m}{\de(2j+2+m\ep)/m}} \indiq_{\beta_M(s)m\in \left[\frac{2j}{m},\frac{2j+2}{m}\right)} 2^{-\frac{m}{\de}(\frac{2}{\beta_M(s)m+\ep/2}-1)}  \,ds \] \nonumber\ala
%&\le 1+ C \E \[ \int_0^t e^{\calP^j_{\sm}}  2^{\frac{m}{\de} \( \frac{2}{(2j+2)/m+\ep} - \frac{2}{\beta_M(s)m+\ep/2} + 1 \) } \indiq_{\beta_M(s)m\in \left[\frac{2j}{m},\frac{2j+2}{m}\right)} \,ds \] \ala
%&\le 1+C \E \[ \int_0^t e^{\calP^j_{\sm}} 2^{m/\de} \,ds \] =   1+C  \int_0^t \E [e^{\calP^j_s}] 2^{m/\de} \,ds \nonumber.
%\end{align*} 
Applying Gronwall's inequality, %applied to $s\mapsto\E[e^{P_k(s\wedge\tau_r -t_{n,\ell})}]$ on the interval $I_{n,\ell}$, 
one obtains that $$\E [e^{P_k(2^{-n}\wedge\tau_r)}] \le e^{\int_0^{2^{-n}} C2^{n/\de} \,\rd s} \le e^C:= K/2.$$ 
%$$\E [e^{\calP^j_{2^{-m}}}] \le e^{ \int_0^{2^{-m}} C2^{m/\de} ds } \le e^{C2^{m/\de}2^{-m}} \le e^C.$$ 
Letting $r\to+\infty$ ends the proof.
\end{proof}
Now we can consider the whole jump process.
%%%%%%%%%%%%%%%%%%%%%%%%%%%%%%%%%%%%%%%%%%%%%%%%%%%
%% Every dyadic interval, moving index %%%%%%%%%%
%%%%%%%%%%%%%%%%%%%%%%%%%%%%%%%%%%%%%%%%%%%%%%%%%%%
\begin{lemma}\label{aux2} Let $K$, $\e_0$ be constants in Lemma \ref{aux1}. For all $\de>1$, $0<\ep<\e_0$,  $n\ge n_0$ (depending only on $\e$ and $G$) and $\ell\in\{0,\cdots,2^n-1\}$, one has
\begin{align*}
\pp \( \sup_{t\in I_{n,\ell}} \lba \int_{\ell 2^{-n}}^{t} \int_{|z|\le 2^{-n/\de}} G(M_{s-},z) \,\wn(\rd s,\rd z)  \rba \ge 2n^2 2^{-\frac{n}{\de\left (\be^{I_{n,\ell},n}_M + \ep\right )}} \) \le K ne^{-2n}.
\end{align*}
\end{lemma}

\begin{proof} It suffices to show this inequality for the first dyadic interval $I:= I_{n,0}$. For other $\ell$, the proof goes along the same lines by an application of Lemma \ref{aux1} for $I_{n,\ell}$.  For each $k\in\{0,\cdots,n-1\}$, write
$$A_k = \llb \sup_{t\le 2^{-n}} \lba \int_0^t\int_{|z|\le 2^{-n/\de}}  G(M_{s-},z) \indiq_{\beta_M(\sm)\in \left[\frac{2k}{n},\frac{2k+2}{n}\right)} \,\wn(\rd s, \rd z) \rba \ge 2n 2^{-\frac{n}{\de \left (\be_M^{I,n}+\ep \right)}} \rrb .$$
%Write
%\begin{equation*}
%A_k =  \( A_k \cap \llb \sup_{s\le 2^{-n}}\beta_M(\sm) < \frac{2k}{n} \rrb \) \cup \( A_k\cap \llb \sup_{s\le 2^{-n}}\beta_M(\sm) \ge \frac{2k}{n} \rrb \) . 
%\end{equation*}
Observe that under the event $$\llb \sup_{s\le 2^{-n}}\beta_M(\sm) < \frac{2k}{n} \rrb$$ the compensated Poisson integral in $A_k$ is zero, thus 
\begin{align}
  \pp(A_k)   &= \pp \( A_k\cap \llb \sup_{s\le 2^{-n}}\beta_M(\sm) \ge \frac{2k}{n} \rrb \) = \pp \( A_k\cap \llb \be_M^{I,n} + \ep \ge \frac{2k+2}{n}+\ep \rrb \) \nonumber\\
& 
\le \pp \( \sup_{t\le 2^{-n}} \lba \int_0^t \int_{|z|\le 2^{-n/\de}} G(M_{s-},z) \indiq_{\beta_M(\sm)\in \left[\frac{2k}{n},\frac{2k+2}{n}\right)} \,\wn(\rd s, \rd z) \rba \ge 2n 2^{-\frac{n}{\de(2k+2+n\ep)/n}} \) \nonumber.
\end{align}
Applying Lemma \ref{aux1} for $I$ implies that  $ 
 \pp(A_k) \le K e^{-2n} . $
Finally, using the inclusion 
\begin{eqnarray}
\llb \sup_{t \leq 2^{-n}} \lba \int_0^t \int_{|z|\le 2^{-n/\de}} G(M_{s-},z) \,\wn(\rd s, \rd z)  \rba \ge 2n^2 2^{-\frac{n}{\de \left (\be^{I,n}_M+\ep \right )}} \rrb
\!\!&\subset&\!\! \bigcup_{k=0}^{n-1}  A_k. \nonumber 
\end{eqnarray}
one deduces the desired inequality. 
%\begin{eqnarray}
%\pp \( \sup_{t \leq 2^{-n}} \!\lba \int_0^t \int_{|z|\le 2^{-n/\be}}  G(M_{s-},z) \,\wn(\rd s, \rd z)  \rba \ge 2m^2 2^{-\frac{m}{\de\left (\be_{0,2^{-m}}^m+\ep\right)}} \) \!\le\! C me^{-2m},\nonumber
%\end{eqnarray} which completes the proof of Lemma \ref{aux2}.
\end{proof}

Let us end the proof of Proposition \ref{LElemme}, using a classical discretization procedure. 

\sk
\begin{preuve}{\it \ of Proposition \ref{LElemme} : }
For  $s<t$ in the unit interval such that $|s-t|\le 2^{-n}$, there exists $\ell\in\{1, \cdots, 2^n\}$ such that $[s,t]\subset I_{n,\ell-1}\cup I_{n,\ell}:=I'_{n,\ell}$. Then 
$$\wh\be_M^{[s,t],n} \ge \be_M^{I_{n,i},n} \ \ \ \mbox{ for } i= \ell-1, \ell. $$ 
Write $$X_n(t)=\int_0^t\int_{|z|\le 2^{-n/\de}} G(M_{s-}, z) \wn(\rd s, \rd z).$$
- Either $s\in I_{n,\ell}$, then triangle inequality entails
$$|X_n(t)-X_n(s)|\le 2\sup_{u\le 2^{-n}} |X_n({t_{n,\ell}}+u)-X_n(t_{n,\ell})|,$$
- or $s\in I_{n,\ell-1}$, then still by triangle inequality
$$|X_n(t)-X_n(s)|\le \sup_{u\le 2^{-n}} |X_n({t_{n,\ell}}+u)-X_n(t_{n,\ell})| + 2\sup_{u\le 2^{-n}} |X_n({t_{n,\ell-1}}+u)-X_n(t_{n,\ell-1})|. $$ 
In any case, $|X_n(t)-X_n(s)|$ is bounded above by two times the maximal displacement of $X_n$ during $2^{-n}$ unit of time starting from time $t_{n,\ell-1}$ and $t_{n,\ell}$. This implies  the inclusion
%\begin{align}
%&\ \ \  2^{\frac{m}{\de\left  (\widehat{\be}_{s,t}^m+\ep \right)}} \lba \int_s^t\int_{C(0,{2^{-m/\de}})} G(\calM_{\sm},z) \,\wn (dudz) \rba \ala
%&\le 2^{\frac{m}{\de\left (\be_{i2^{-m},(i+1)2^{-m}}^m+\ep \right)}} \sup_{t \leq 2^{-m}} \lba \int_{i2^{-m}}^{i2^{-m}+t} \int_{C(0,{2^{-m/\de}})} G(\calM_{s-},z) \,\wn(dsdz) \rba  \ala
%& \ \ + 2\cdot 2^{\frac{m}{\de\left (\be_{(i-1)2^{-m},i2^{-m}}^m+\ep \right)}} \sup_{t \leq 2^{-m}} \lba \int_{(i-1)2^{-m}}^{(i-1)2^{-m}+t} \!\!\int_{C(0,{2^{-m/\de}})} \hspace{-5mm} G(\calM_{s-},z) \,\wn(dsdz) \rba \nonumber
%\end{align}
\begin{multline*}
 \llb \sup_{\overset{|s-t|\le 2^{-n}}{s<t\in\zu}}   2^{\frac{n}{\de\left  (\widehat{\be}_M^{[s,t], n}+\ep\right)}} \lba \int_s^t\int_{|z|\le {2^{-n/\de}}}  G(M_{\sm},z) \,\wn (\rd u, \rd z) \rba \ge 8 n^2 \rrb \subset \\
 \bigcup_{\ell=0}^{2^n-1} \llb 2^{\frac{m}{\de\left ({\be}_M^{I_{n,\ell}, n}+\ep\right)}} \!\!\!\!\sup_{u \leq 2^{-n}} \lba \int_{t_{n,\ell}}^{t_{n,\ell}+u} \int_{|z|\le 2^{-n/\de}}  G(M_{s-},z) \,\wn(\rd s,\rd z)  \rba \ge 2n^2 \rrb 
\end{multline*}
which, together with Lemma \ref{aux2}, completes the proof. 
%\begin{align}
%\nonumber&  \pp \( \sup_{\overset{|s-t|\le 2^{-m}}{s,t\in\zu}}   2^{\frac{m}{\de\left  (\widehat{\be}_{s,t}^m+\ep \right)}} \lba \int_s^t\int_{C(0,{2^{-m/\de}})} G(\calM_{\sm},z) \,\wn (dudz)   \rba \ge 6 m^2 \) \\
%& \le 2^m C m e^{-2m} \le Ce^{-m} \nonumber.
%\end{align}
 \end{preuve}

%%%%%%%%%%%%%%%%%%%%%%%%%%%%%%%%%%%%%%%%%%%%%%%%%%%%%%%%%%%%%%%%%%%%%%
%%%%%%%%%%%%%%%%%%%%%%%%%%%%%%%%%%%%%%%%%%%%%%%%%%%%%%%%%%%%%%%%%%%%%%
\section{H\"older exponent}%%%%%%%%%%%%%%%%%%%%%%%%%%%%%%%%%%%%%%%%%%%%
%%%%%%%%%%%%%%%%%%%%%%%%%%%%%%%%%%%%%%%%%%%%%%%%%%%%%%%%%%%%%%%%%%%%%%
%%%%%%%%%%%%%%%%%%%%%%%%%%%%%%%%%%%%%%%%%%%%%%%%%%%%%%%%%%%%%%%%%%%%%%

In order to compute the multifractal spectrum, we have to investigate the pointwise H\"older exponent of the solution $M$ to \eqref{jd}.  Two situations may occur :
\begin{itemize}
\item $|\sigma|$ is bounded below away from zero.  As it turns out,  the Lebesgue integral in \eqref{jd} is smoother at every point than the Brownian integral  which has constant H\"older exponent.  So the former does not affect the H\"older exponent of $M$. It then suffices to determine the exponent of the compensated Poisson integral.
\item $\sigma$ is identically zero. Without the Brownian integral, the H\"older regularity of $M$ is determined by the rougher one  among the compensated Poisson integral and the Lebesgue integral (drift). As is said in the introduction, the  H\"older exponent of the drift is unknown and out of the scope of this work. However, by imposing more regularity on the coefficients, we show that the compensated Poisson integral dominates the H\"older regularity of $M$.
\end{itemize}
In any case, the non-compensated Poisson integral is not an issue because it is piecewise constant with finite number of jumps in any finite interval.  

\mk

Let us first state a result on the H\"older regularity of the Brownian integral.
%%%%%%%%%%%%%%%%%%%%%%%%%%%%%%%%%%%%%%%%%%%%%%%%%%%%%%%%%%%%%%%%%%%%%%%
\begin{proposition} \label{difexp} Assume that (H1)-(H3) hold with non trivial $\sigma$. Let $M$ be the solution to \eqref{jd} and $C_t = \int_0^t \sigma(M_{s-})\,\rd B_s$.  Then almost surely,  for all $t\in [0,1]$, $H_C(t) = \frac{1}{2}.$
\end{proposition}
%%%%%%%%%%%%%%%%%%%%%%%%%%%%%%%%%%%%%%%%%%%%%%%%%%%%%%%%%%%%%%%%%%%%%%%
By Dambis-Dubins-Swartz theorem, $C$ (which is a continuous martingale) can be written as a Brownian motion subordinated in time. The point here is that the subordinated process is bi-Lipschitz continuous. The Hölder regularity of Brownian motion can thus be inherited by the martingale $C$. This is somewhat classical, we will include a proof for completeness in Appendix B.

%%%%%%%%%%%%%%%%%%%%%%%%%%%%%%%%%%%%%%%%%%%%%%%%%%%%%
\begin{remark} The non degenerate condition on $\sigma$ cannot be dropped. Indeed, when $\sigma(M_t) = 0$, the process $C$ may gain more regularity at $t$ and the computation of $H_{C}(t)$   involves the regularity of $\sigma(M_{\cdot})$ at time $t$.\end{remark}

\mk

To state the main result of this section, we need some notations. 
%The Poisson random measure $N$ can be derived from a Lévy process $\calL$ with characteristic triplet $(0, 0, \pi(dz))$, meaning that there are no Brownian component and no drift. For all $s,t\in\rr^+$ and every Borel set $A\in\cB(C(0,1))$, denote $N([s,t],A) = \sharp\{ u\in[s,t] : \Delta \cL_u\in A \}$. 
%Then,  almost surely, $N$ is a Poisson random measure with intensity $dt\otimes \pi(dz)$. 
Define the point system\begin{equation}
 \label{defP}
 \calP=(T_n,Z_n)_{n\ge 0}
 \end{equation} where $(T_n,Z_n)$ is the Poisson point process associated with the Poisson measure $N(\rd t, \rd z)$ so that  
 $$N(\rd t, \rd z)=\sum_{n\ge 1} \delta_{(T_n,Z_n)}(\rd t, \rd z).$$
 We can assume that $(|Z_n|)_{n\in \nn}$ forms a decreasing sequence by rearrangement. By properties of Poisson integral, the set of locations of the jumps  $J= \{ T_n : n\in\nn \}$ and for each $n$,  $\Delta_{T_n} = G(M_{T_n -},Z_n)$ where $\Delta_t := M_t - M_{t-}$. See \cite[Section 2.3]{applebaum2009} for details. 
 
The approximation rate $\de_t$ by $\calP$ describes how close to the jump points $T_n$ a point $t$ is.  Intuitively, the larger $\delta_t$ is, the closer to large jumps $t$ is.
 %%%%%%%%%%%%%%%%%%%%%%%%%%%%%%%%%%%%%%%%%%%%
\begin{definition} \label{defsys} 
The approximation rate of $t\in \rr^+$ by $\calP$ is defined by
\begin{eqnarray}
\de_t = \sup \{ \de\ge 1 : |T_n-t|\le |Z_n|^{\de}  \mbox{ for infinitely many } n \} .\nonumber
\end{eqnarray}
\end{definition}
%%%%%%%%%%%%%%%%%%%%%%%%%%%%%%%%%%%%%%%%%%%%%

 We can state the main result of this section. The random approximation rate  plays a key role.% when investigating the pointwise regularity of $M$, as stated by the main result in this section.

%Given a system of point $S = \( (t_n,l_n)_{n\ge 0}\) $ satisfying the covering property
%\begin{equation}\label{coverprop}
%\rr^+ \subset \limsup_{n\to +\infty}B(t_n,l_n),
%\end{equation} where $t_n\in \rr^+$ and $(l_n)_{n\ge 0}$ is a strictly positive sequence decreasing to $0$

\begin{theorem}
\label{exponent} 
Assume that (H1)-(H4) hold. 
\begin{enumerate}
\item  If $|\sigma|$ is bounded below away from zero, then almost surely, $$\forall \, t\notin J, \ \ \ \ H_M(t) = \frac{1}{\de_t\beta_M(t)}\wedge \frac{1}{2}.$$
\item  If  $\sigma$ is identically zero and (H5) holds, then almost surely, $$ \forall \, t\notin J,  \ \ \   H_M(t) = \frac{1}{\de_t\beta_M(t)}.$$
\end{enumerate}
\end{theorem}

\section{Proof of Theorem \ref{exponent}}\label{sectionproofj}

\subsection{Preparations}
 First let us introduce a family of limsup sets that are proved to be relevant in the regularity study of compensated Poisson integral, see \cite{jaffard1999levy}.  Set for each $\de\ge 1$,
\begin{align*}
A_{\de} = \limsup_{n\rightarrow +\infty} B(T_n,|Z_n|^{\de})
\end{align*}
The following covering property for the system $\calP$ of time-space points is well known.

%%%%%%%%%%%%%%%%%%%%%%%%%%%%%%%%%%%%%%%
\begin{proposition}\label{cover01}
With probability one, $\zu\subset A_1$.
\end{proposition}
\begin{proof}
Recall that the Poisson measure $N$ has intensity $\rd t\,\pi(\rd z)$ where $\pi(\rd z) = \rd z/z^2$. Using Shepp's theorem \cite{shepp1972randomcovering} (and a integral test by Bertoin \cite{bertoin1994integraltestforcovering}), it suffices to prove that 
$$I=\int_0^1 \exp \( 2\int_t^1 \pi((u,1)) \,du \) dt = +\infty.$$
But $\pi((u,1)) = u^{-1}-1$, so that $I=\int_0^1 e^{2(t-1-\ln t)}dt   = + \infty$.
\end{proof}
%%%%%%%%%%%%%%%%%%%%%%%%%%%%%%%%%%%%%%%

It follows that almost surely the approximation rate $\de_t$ of $t$ by the system of points $\cP$ (see Definition \ref{defsys}) is  well-defined,  always greater than or equal to 1, and random because it depends on $N$.  Using merely the definition of $\de_t$ and the covering property of $\cP$, one can obtain an upper bound for the H\"older exponent of the compensated Poisson integral 
$$X_t = \int_0^t\int_{|z|\le 1} G(M_{s-},z) \wn(\rd s, \rd z). $$

%%%%%%%%%%%%%%%%%%%%%%%%%%%%%%%%%%%%%%%%%%%%%%%%%%%%%%%%%%%%%%%%%%%%%%%
\begin{proposition}\label{jumpexp} With probability one, $\forall t\in[0,1]$, $ H_X(t) \le \frac{1}{\beta_M(t)\de_t}$ and $H_M(t) \le \frac{1}{\beta_M(t)\de_t}$.
\end{proposition}
%%%%%%%%%%%%%%%%%%%%%%%%%%%%%%%%%%%%%%%%%%%%%%%%%%%%%%%%%%%%%%%%%%%%%%%

The proof is based on two lemmas. The first is observed by Jaffard \cite[Lemma 1]{jaffard1997old} who found the importance of the dense jumps in the study of local regularity of functions.

%%%%%%%%%%%%%%%%%%%%%%%%%%%%%%%%%%%%%%%%%%
%% Jaffard's Lemma %%%%%%%%%%%%%%%%%%%%%%%
%%%%%%%%%%%%%%%%%%%%%%%%%%%%%%%%%%%%%%%%%%

\begin{lemma}[\cite{jaffard1997old}]\label{old}
Let $f:\rr\mapsto \rr$ be a c\`adl\`ag function discontinuous on a dense set of points , and let $t\in\rr$. Let $(t_n)_{n\ge 1}$ be a real sequence converging to $t$ such that, at each $t_n$, $|f(t_n) - f(t_n -)| = z_n>0$. Then 
$$H_f(t) \le \liminf_{n\rightarrow +\infty} \frac{\ln z_n}{\ln |t_n - t|}.$$
\end{lemma}

The second lemma establishes a first link between the pointwise regularity and the approximation rate. 

%%%%%%%%%%%%%%%%%%%%%%%%%%%%%%%%%%%%%%%%%
%% Intermediate step for upper bound %%%%
%%%%%%%%%%%%%%%%%%%%%%%%%%%%%%%%%%%%%%%%%

\begin{lemma}\label{up}
For all $\de\ge 1$, almost surely 
\begin{equation}
  \forall\,t \in A_{\de}, \ \ \ \ \  H_X(t) \le \frac{1}{\beta_M(t)\de}    . \label{upeq}
\end{equation}
\end{lemma}
\begin{proof}
Recall that almost surely the set of jump times is  
$$J = \{T_n : n\in\nn_*\}$$
 and  at $T_n$, the jump size of $X$ is $G(M_{T_n -},Z_n)$.  If $t\in J$, the desired inequality is trivial. Consider $t\in A_\de\setminus J$. Necessarily, $t$ is a continuous time of $M$ and there is an infinite number of $n$ such that 
\begin{equation}\label{balls} |T_n-t|\le |Z_n|^{\de} \end{equation} 
with $|Z_n|$ decreasing to zero.  Lemma \ref{old} applied to the process $X$ with  the jumps satisfying \eqref{balls} and the triangle inequality imply
\begin{align*}
H_X(t) &\le \liminf_{n\to +\infty} \frac{\ln |G(M_{T_n-}, Z_n)|}{\ln |T_n-t|} \\ &\le \liminf_{n\to +\infty} \frac{\ln |G(M_{T_n-}, Z_n)|}{\de \ln |Z_n|} \ala
&\le   \limsup_{n\to +\infty} \frac{-|\ln |G(M_{T_n -},Z_n)| - \ln|G(M_t,Z_n)||}{\de \ln|Z_n|} + \liminf_{n\to +\infty} \frac{\ln |G(M_t, Z_n)|}{\de \ln |Z_n|} 
\end{align*}
By the local Lipschitz condition in (H4), there exists a finite constant $C$ (that depends on the maximum of $M$ in $[0,1]$) such that the first term is bounded above by $(C/\de)\limsup|M({T_n-})-M(t)|$ which is zero by the continuity of $M$ at $t$. The second equals to $1/(\be_M(t)\de)$ by the asymptotically stable-like condition in (H4), as desired.
%
%where we used \eqref{balls} for the second inequality, the Lipschitz condition of $G$ for the last inequality and the continuity of $M$ at time $t$ for the last equality. 
%For all $t\in J$, $H_\calZ(t) = 0$, which completes the proof.
\end{proof}

%
%$$ 
%\le \frac{1}{\beta_M(t)\de} + \frac{C}{\de} \liminf_{n\to +\infty} |M_{T_n -} - M_t|  = \frac{1}{\beta_M(t)\de}.$$

%%%%%%%%%%%%%%%%%%%%%%%%%%%%%%%%%%%%%%%%%%%%%%%%%%%%%%%%%%%%%%%%%%%%%%%%%%
\sk
\begin{preuve}{\it \ of Proposition \ref{jumpexp} : }
It follows from Lemma \ref{up} that a.s., for all rational number $\de\ge 1$, \eqref{upeq} holds. Using the monotonicity of $\de\mapsto A_{\de}$ and the density of rational numbers in $[1,+\infty)$, we deduce that  almost surely \eqref{upeq} holds for all $\de\ge 1$.  Let $t\in [0,1]$, two cases may occur.

If $\de_t<+\infty$,    then  $t\in A_{\de_t-\ep}$, for every $\ep>0$. Hence,  $H_X(t)\le \frac{1}{\beta_M(t)(\de_t-\ep)}$ as a  consequence of Lemma \ref{up}.  Letting $\ep$ tend to $0$, we obtain the result.

If $\de_t = + \infty$, then $  t\in\bigcap_{\de\ge 1} A_\de $, meaning that $t\in B(T_n, |Z_n|^{\de})$  for infinitely many integers $n$, for all $\de\ge 1$. We deduce by Lemma \ref{up} that $H_X(t)\le \frac{1}{\beta_M(t)\de}$, for all $\de\ge 1$, thus $H_X(t)=0$, as desired.  

To deduce that $1/(\de_t\be_M(t))$ is also an upper bound for $H_M(t)$, one simply  remarks that for any $t\notin J$, the approximation rate $\de_t$ is the same for $M$ and $X$.
\end{preuve}
%%%%%%%%%%%%%%%%%%%%%%%%%%%%%%%%%%%%%%%%%%%%%%%%%%%%%%%%%%%%%%%%%%%%%%%%%%%

\sk
 We need the following two lemmas, whose proofs are elementary and are left to Appendix A. 
%%%%%%%%%%%%%%%%%%%%%%%%%%%%%%%%%%%%%%%%%%%%%%%%%%%%
\begin{lemma}  \label{techlemma} Let $f,g : \rr \rightarrow \rr$, ${\ds F(x) = \int_0^{x} f(y)\,dy}$ and $x_0\in\rr$.
\begin{itemize}
\item[(i)] $ H_{F}(x_0)\ge H_f(x_0)+1$. 
\item[(ii)]  If $g\in C^k(\rr)$ with $k=\inf\{\ell\in\nn: H_f(x_0)\le \ell \}$,  then $H_{g\circ f}(x_0)\ge H_f(x_0)$.
\end{itemize}
\end{lemma}

\begin{lemma}\label{techlemma2} Suppose (H4). For any $m\in \nn^*$,  there is a finite constant $c_m$ such that for $x,y\in D_m = \{u\in\rr: |u|\le m \mbox{ and } \be(u) \le 1-1/m \}$. 
\begin{align*}
|\wt G(x)|\le c_m(1+|x|) \ \ \mbox{ and } \ \ \ |\wt G(x) - \wt G(y)|\le c_m |x-y|,
\end{align*}
where $\wt G(x)= \int_{|z|\le 1} G(x,z)\,\pi(\rd z)$.
\end{lemma}

\subsection{An important observation}
We intend to show that the upper bound obtained in Proposition \ref{jumpexp} is  optimal when $\sigma=0$. Let us first describe the configuration of the jumps around a time $t$ outside those limsup sets $(A_\de, \de\ge 1)$. Let $t\notin A_{\de}\cup J$, then there exists a random integer $n_0$, such that 
\begin{equation}\label{eq not A_de}
\forall\, n\ge n_0,\ \ \ |T_n - t|\ge |Z_n|^{\de}.
\end{equation} 
Let $s>t$ sufficiently close to $t$ such that $[t,s]$ does not contain those $T_n$ which violate \eqref{eq not A_de}. It is possible because the cardinality of such $T_n$ is finite.  For each $s$, there exists a unique integer $j$ such that $2^{-j-1}\le |s-t| < 2^{-j}$. Assume that $T_n\in [t,s]$, then  $2^{-j} > |t-s| \ge |T_n - s| \ge |Z_n|^{\de}$, so that $|Z_n| \le 2^{-j/\de}$. This means that in an interval of length $2^{-j}$ with one extreme point in the complement of $A_{\de}\cup J$, there is no jump whose corresponding Poisson jump size larger than $2^{-j/\de}$. Therefore, to consider the increment of the compensated Poisson integral $X$ near such time $t$, one can split the increment of $X$ into two parts:
\begin{equation}\label{decomposition} 
X_s-X_t = 
\int_t^s \int_{|z|\le 2^{-j/\de}} G(M_{\um},z)\,\wn(\rd u, \rd z) + \int_t^s\int_{2^{-j/\de}<|z|\le 1} G(M_{\um},z)\,\wn(\rd u, \rd z). 
\end{equation}
and the second integral is in fact a Lebesgue integral. This decomposition shows why Proposition \ref{LElemme} is so important. 

%%%%%%%%%%%%%%%%%%%%%%%%%%%%%%%%%%%%
% heuristics from Lévy process%%%%%%%
%%%%%%%%%%%%%%%%%%%%%%%%%%%%%%%%%%%%
%\begin{remark}
%To obtain the local regularity of a L\'evy process $L$, Jaffard \cite{j} cut the process into a sequence of process $L^j$ whose jump size is of a given order (for instance, a constant power of $2^{-j}$). Since we study local properties and we drop the assumption of stationary increment, our cutting procedure for the jump process $Z$ will be {\em local}, namely, around a given point $t\notin \cJ$, jump size of the $j$th band $Z^j$ will be a power of $2^{-j}$ where the power depends on the value of $M_t$.
%\end{remark}

\subsection{Proof of Theorem \ref{exponent} (ii)}\label{sec5.3}
When the diffusion coefficient is identically zero, 
\begin{equation}\label{jdsansbm}
M_t = M_0+ \int_0^t b(M_u)\,\rd u + \int_0^t \int_{|z|\le 1} G(M_{u-},z) \,\wn(\rd u, \rd z) + \int_0^t \int_{|z|>1} F(M_{s-}, z) N(\rd s, \rd z).
\end{equation} 
Set in the sequel
$$Y_t = \int_0^t b(M_s)\,\rd s \mbox{ and } Y'_t = \int_0^t \wt G(M_s)\,\rd s,$$
recall that $\wt G(x)= \int_{|z|\le 1} G(x,z)\pi(\rd z)$. We distinguish two cases which correspond to different conditions in (H5).

\mk

\subsubsection{First condition in (H5):}\label{subsubsec} $b\in C^2(\rr)$ and $\inf\{\be(x) : x\in \rr\}\ge 1/2$. 

 \sk
Applying Lemma \ref{techlemma} to the Lebesgue integral $Y$ in \eqref{jdsansbm} implies that  a.s.
$$\mbox{ for all t }\in [0,1]\setminus J, \ \ \ H_Y(t) \ge  H_{b\circ M}(t)+1 \ge H_M(t)+1,$$
where we used the assumption $b\in C^2(\rr)$ and the upper bound $H_M(t)\le 1/\be_M(t) \le 2$ ($\be$ is bounded below by $1/2$) obtained in Proposition  \ref{jumpexp}. 
This, together with the fact that the non compensated Poisson integral in \eqref{jdsansbm} is piecewise constant with finite number of jumps in $[0,1]$, entails that a.s.
$$\mbox{ for all t }\in [0,1]\setminus J, \ \ \ H_M(t) =  H_X(t).$$
It remains to show that a.s. for each continuous time $t$ of $M$,  the H\"older exponent of $X$ is $1/(\de_t\beta_M(t))$.   This value is an upper bound for $H_X(t)$ due to Proposition \ref{jumpexp}. To show that it is also a lower bound, it suffices to show 
%The upper bound of $H_X(t)$ is obtained in Proposition \ref{jumpexp}, so it remains us  to get the lower bound, which will be deduced from the following  property:
\begin{equation}\label{enough}
\forall\,\delta>1,\, \forall\, \ep>0, \, \text{ almost surely, } \forall\, t\notin J\cup A_{\de}, \,H_X(t)\ge \frac{1}{\delta(\beta_M(t)+\ep)}. 
\end{equation} 
Indeed, a routine argument by density of rational points and monotonicity of events, together with the definition of approximation rate, entail the sufficiency. 

%%%
% Details
% assume  that \eqref{enough} holds true. This implies   that, almost surely,  for all rational pair $\ep>0$ and  $\delta>1$,  one has  $ H_\calZ(t)\ge \frac{1}{\delta(\beta_M(t)+\ep)}$  for all points $t\notin J\cup A_{\de}$.   The monotonicity of the mapping $\de\mapsto A_{\de}$ yields that if $\delta'>\delta$, $(t\notin A_{\de} ) \Rightarrow (t\notin A_{\de'})$. One deduces that almost surely, for every $ \delta>1$ real,  for every  rationals    $\delta'>\de$  and $\ep>0$,    if $ t\notin J\cup A_{\de}$, the exponent $ H_\calZ(t) $ satisfies $ H_\calZ(t)\ge \frac{1}{\de'(\beta_M(t)+\ep)}.$
% Using the density of the rational numbers in $\rr$ and taking $\ep$ arbitrarily small in $\mathbb{Q}$ yields that almost surely, $$\forall\, \delta>1, \ \ \forall\, t\notin J\cup A_{\de},\ \  H_\calZ(t)\ge \frac{1}{\de  \beta_M(t) }.$$ 
%
%We deduce the lower bound for $H_\calZ(t)$. Let $\ep'>0$. If $t\notin J$ and $\de_t < +\infty$, one has $t\notin A_{\delta_t + \ep'}$ by the definition of the approximation rate $\de_t$.  hence $H_\calZ(t) \ge \frac{1}{(\de_t+\ep')\beta_M(t)}$. If $t\notin J$ but $\de_t = +\infty$, then $\frac{1}{(\de_t+\ep') \beta_M(t)} = 0$, the desired inequality is trivial.  Letting $\ep'\to 0$ yields that a.s., $\forall\,t\notin J$, $H_\calZ(t)\ge \frac{1}{\de_t\beta_M(t)}$.
%%%%%
 
\medskip

Now we prove \eqref{enough}. Applying the technical estimate Proposition \ref{LElemme} and Borel-Cantelli lemma,  one obtains that for any $\ep>0$, $\de >1$, almost surely, for all $n$ larger than some $n_0$,
\begin{equation*}
 \sup_{\overset{|s-t|\le 2^{-n}}{s<t\in\zu}} \lba 2^{\frac{n}{\de \left (\widehat{\be}_{M}^{[s,t],n}+\ep/3 \right)}}  \int_s^t\int_{|z|\le 2^{-n/\de}}  \hspace{-2mm} G(M_{\um},z) \,\wn (\rd u,\rd z) \rba \le 8 n^2
\end{equation*}
%Fix $\omega$. 
In particular, for each $t\notin J\cup A_{\de}$ and $s\in B(t,2^{-n_0})$, there is a unique $n\ge n_0$ such that $2^{-n-1} \le |s-t|< 2^{-n}$ and 
\begin{equation}\label{step3eq2}
\lba 2^{\frac{n}{\de \left (\widehat{\be}_{M}^{[s,t],n}+\ep/3 \right) }}  \int_s^t\int_{|z|\le 2^{-n/\de}}  G(M_{\um},z) \,\wn (\rd u,\rd z) \rba \le 8 n^2.
\end{equation}
Using \eqref{step3eq2} and the continuity of $\be_M$ at $t$, the "large" jumps removed increment of $X$
\begin{equation}
\lba  \int_s^t\int_{|z|\le 2^{-n/\de}}  \hspace{-2mm} G(M_{\um},z) \,\wn (\rd u,\rd z) \rba \le |s-t|^{\frac{1}{\delta(\beta_M(t)+\ep/2)}} \( \ln\frac{1}{|s-t|}\) ^2 \label{beta12eq1}
\end{equation}
if $n_0$ is large enough. We enlarge the value of $n_0$ if necessary to ensure this. %where we used the continuity of $\beta:= \widetilde\beta\circ M$ at $t$ (because $\widetilde \beta$ is Lipschitz  and $M$ is continuous at $t$)   and  the fact that $\widehat{\beta}^m_{s,t}\le \beta_M(t)+\ep$ for large $m$.

Recalling the discussion in last subsection and the decomposition  \eqref{decomposition}, in the interval $[s,t]$ with $t\notin J\cup A_{\de}$ and $2^{-n-1}\le |s-t|< 2^{-n}$, there is no jump time whose corresponding jump size is larger than $2^{-n/\de}$, namely $N([s,t]\times\{z: 2^{-n/\de}<|z|\le 1\})=0$. Hence, 
\begin{equation}\label{beta12eq2}
 \int_s^t\int_{2^{-n/\de}< |z|\le 1}  G(M_{\um},z) \,\wn (\rd u,\rd z) = -\int_s^t\int_{2^{-n/\de}< |z|\le 1}  G(M_{\um},z) \,\pi(\rd z)\rd u.
\end{equation}

Two situations may occur. 
\begin{enumerate}
\item $\be_M(t)\de\ge 1$. The desired lower bound for $H_X(t)$ is less than $1$, hence, one only needs to consider the constant polynomial in the definition of H\"older exponent. We split the right-hand side integral in \eqref{beta12eq2} into two parts. Using the one-sided uniform bound in the asymptotically stable-like assumption, 
\begin{align}\label{beta12eq3}
\lba \int_{2^{-n/\de}< |z|\le z(\e)}  G(M_{\um},z) \,\pi(\rd z)\rba \le \int_{2^{-n/\de}< |z|\le z(\e)} |z|^{1/(\be_M(u-)+\e/3)} \,\pi(\rd z)
\end{align}
which is bounded above by $\max(c2^{(n/\de)(1-1/(\be_M(t)+\e/2))},1)$ with $c$ a finite constant that depends only on $G$. By Cauchy-Schwarz inequality,
\begin{align}\label{beta12eq4}
\lba \int_{z(\e)<|z|\le 1}  G(M_{\um},z) \,\pi(\rd z)\rba ^2 \le  2 \cdot (z(\e)^{-1}-1) \int_{z(\e)< |z|\le 1} |G(M_{\um},z) |^2 \,\pi(\rd z) 
\end{align}
where the right-hand side integral is bounded above by $C(1+|M_{u-}|^2)\le 2C(1+|M_{t}|^2)$ due to linear growth condition (H1) and the continuity of $M$ at $t$. 
Combining \eqref{beta12eq3}-\eqref{beta12eq4}, one obtains that
\begin{align}\label{beta12eq5}
\lba \int_s^t\int_{2^{-n/\de}< |z|\le 1}  G(M_{\um},z) \,\pi(\rd z)\rd u \rba \le C \max(|s-t|^{1-\frac{1}{\de}+\frac{1}{\de(\be_M(t)+\e/2)}}, |s-t|).
\end{align}
 Using \eqref{beta12eq1}, \eqref{beta12eq2} and \eqref{beta12eq5}  yields 
\begin{align*}
|X_t - X_s|\le C|s-t|^{\frac{1}{\delta(\beta_M(t)+\ep)}} \( \ln\frac{1}{|s-t|}\) ^2
\end{align*}
which proves $H_X(t)\ge 1/(\de(\be_M(t)+ \e))$.
\item $\be_M(t)\de<1$, thus $\be_M(t)+\e<1$ for sufficiently small $\ep$. The desired lower bound for $H_X(t)$ is now a number in $(1,2]$, due to the assumption $\inf\{\be(x): x\in\rr\}\ge 1/2$. To study $H_X(t)$,  one has to subtract a linear polynomial from $X_s$. Using the decomposition \eqref{decomposition} and the observation \eqref{beta12eq2}, the quantity
\begin{align*}
|X_{s}-X_t + (s-t)\wt G(M_t)|
\end{align*}
is bounded above by the sum of 
\begin{align*}
I_1 &= \lba \int_t^{s}\int_{|z|\le 2^{-n/\de}} G(M_{u-},z)\,\wt N(\rd u, \rd z) \rba ,\\
 I_2 &= \lba \int_t^{s}\int_{2^{-n/\de}< |z|\le 1} G(M_{u-},z)\,\pi(\rd z)\rd u - \int_t^{s}\int_{2^{-n/\de}< |z|\le 1} G(M_t,z)\,\pi(\rd z)\rd u \rba , \\
I_3 &= \lba \int_t^{s}\int_{|z|\le 2^{-n/\de}} G(M_t,z)\,\pi(\rd z)\rd u \rba .
\end{align*}
By Lemma \ref{techlemma2}, one has
\begin{align*}
I_2 \le \int_s^t |\wt G(M_{u-}) - \wt G(M_t)| \,\rd u \le C(1+|M_t|)\int_s^t|M_{u}-M_t|\,\rd u. 
\end{align*}
where $C=C(\ep)$ depends on $\ep$. Recall that $1<H_M(t)\le 1/\be_M(t)\le 2$, so there is a polynomial $P$ of degree at most $1$ such that
\begin{align*}
|M_s-P(s-t)| = O(|s-t|^{H_M(t)-\e})
\end{align*} 
as $|s-t|\to 0$.  
\begin{itemize}
\item Either $P$ is of degree zero, then one has $|M_u - M_t| \le C|u-t|^{H_M(t)-\e}$
so that
\begin{align*}
 I_2 \le C(1+|M_t|)|s-t|^{H_M(t)+1-\e},
\end{align*}
\item or $P$ is of degree $1$,  then one has $
|M_u - M_t| = O(|u-t|)$ so that 
\begin{align*}
 I_2 \le C(1+|M_t|)|s-t|^{2}.
\end{align*}
\end{itemize}
Now we bound from above $I_3$. By asymptotically stable-like assumption (one-sided uniform bound), the integral over $z$ inside $I_3$ is bounded above by
\begin{align*}
\int_{|z|\le 2^{-n/\de}} |z|^{1/({\be_M(u-)+\e/2})} \,\pi(\rd z) \le C 2^{-(n/\de)(1/(\be_M(u-)+\e/3)-1)}
\end{align*}
so that 
\begin{align*}
I_3 \le C_2|s-t|^{1+ \frac{1}{\de}(\frac{1}{\be_M(t)+\e/2}-1)} =C |s-t|^{1 - \frac{1}{\de}+ \frac{1}{\de(\be_M(t)+\e/2)}} .
\end{align*}
Using \eqref{beta12eq1} and $H_M(t)=H_X(t)$, together with the above estimates, yields that 
\begin{align*}
|X_{s}-X_t + (s-t)\wt G(M_t)| \le C |s-t|^{\frac{1}{\delta(\beta_M(t)+\ep)}} \( \ln\frac{1}{|s-t|}\) ^2.
\end{align*}
This entails the desired lower bound for $H_X(t)$.
\end{enumerate}

\subsubsection{Second condition in (H5):}  $\sup\{k\in\nn: b \mbox{ and } \wt G \in C^{k}(\rr)\} \ge \sup\{1/\be(x), x\in\rr\}$.

\sk
Suggested by Lemma \ref{techlemma2}, when $\be_M(t)<1$, one should be able to write locally the increment of $X$ as a non-compensated Poisson integral minus its compensator.  We need the following lemma to show this rigorously. 
 \begin{lemma}\label{fivar}
Almost surely, for any $t\in[0,1]$,
\begin{align*}
 \int_0^t\int_{|z|\le 1} |G(M_{u-},z)|\indiq_{\be_M(u-)<1} \,N(\rd u, \rd z) < +\infty.
\end{align*} 
\end{lemma} 
\begin{proof}
It suffices to show that for any $\eta>0$,  a.s.
\begin{align*}
 \int_0^1\int_{|z|\le 1} |G(M_{u-},z)|\indiq_{\be_M(u-)<1-\eta} \, N(\rd u, \rd z)  < +\infty.
\end{align*}
%Let $\Omega_r$ be the event $\{\sup_{t\in[0,1]}|M_t|\le r\}$. Due to c\`adl\`ag property of the sample paths of $M$, $\pp(\Omega_r)\to 1$ as $r\to +\infty$.  
Define the stopping times $\tau_r = \inf\{t\ge 0: |M_t|>r \}$. Observe that a.s. $\lim_{r\to\infty}\tau_r\to\infty$. One has 
\begin{align*}
\E \left[ \int_0^{1\wedge\tau_r}\int_{|z|\le 1}  |G(M_{u-},z)| \indiq_{\be_M(u-)<1-\eta} \,\pi(\rd z)\rd u \right] <+\infty.
\end{align*}
Indeed, one uses the one-sided uniform bound in (H4) and Lemma \ref{techlemma2} to bound from above the integral on the domain $\{z: |z|\le z(\e)\}$, then Cauchy-Schwartz  inequality and linear growth condition to bound the integral on the domain $\{z: z(\e)<|z|\le 1\}$.  Therefore, a.s. for all rational $r\ge 0$,
\begin{align*}
\int_0^{1\wedge\tau_r}\int_{|z|\le 1}  |G(M_{u-},z)| \indiq_{\be_M(u-)<1-\eta} \, N(\rd u, \rd z) < +\infty
\end{align*}
which entails the result.
\end{proof}

Now consider $t\in[0,1]\setminus J$. If $\be_M(t)<1$, then for any $s$ in a small neighborhood of $t$, $\be_M(s)<1$, so that
\begin{align*}
\int_s^t\int_{|z|\le 1} G(M_{u-},z)\,\wt N(\rd u, \rd z) = \int_s^t\int_{|z|\le 1} G(M_{u-},z)\indiq_{\be_M(u-)<1}\,\wt N(\rd u, \rd z)
\end{align*}
which, by Lemma \ref{fivar} and $\sup\{\be_M(u): u\in[s,t]\}<1$, is   
\begin{align*}
\int_s^t\int_{|z|\le 1} G(M_{u-},z)\,N(\rd u, \rd z) - \int_s^t\int_{|z|\le 1} G(M_{u-},z)\,\pi(\rd z)\rd u
\end{align*}
and both integrals are finite. We thus have established another representation for $M$ around $t$ :
\begin{align*}
M_t - M_s = (X'_t - X'_s) - (Y'_t-Y'_s) + (Y_t - Y_s)
\end{align*}
where 
\begin{align*}
X'_s= \int_0^s \int_{|z|\le 1} G(M_{u-},z)\,N(\rd u, \rd z).
\end{align*}
An application of Lemma \ref{techlemma} and the regularity assumption on $\wt G$ yields that almost surely,
\begin{align*}
\mbox{ for all } t\notin J \mbox{ with } \be_M(t)<1, \ \ \  H_M(t) = H_{X'}(t).
\end{align*}
We proceed to show that almost surely, 
\begin{equation}\label{expX'}
\mbox{ for all } t\notin J \mbox{ with } \be_M(t)<1, \ \ \  H_{X'}(t)=\frac{1}{\de_t\be_M(t)}.
\end{equation}
Let us stress the fact that no regularity assumption on $b$ or $\wt G$ is needed to show this. Following the same lines in the proof of Proposition \ref{jumpexp} (since $X$ and $X'$ are constructed using the same Poisson time-space points), one has almost surely 
\begin{align*}
\mbox{ for all } t\notin J \mbox{ with } \be_M(t)<1,  \ \ \ H_{X'}(t)\le \frac{1}{\de_t\be_M(t)}.
\end{align*}
It remains to show \eqref{enough} with $X$ replaced by $X'$. For all $s$ that is sufficiently close to $t$, there is a unique $n$ such that $2^{-n-1}\le |s-t| < 2^{-n}$.
Applying \eqref{beta12eq2}, one has
\begin{align*}
|X'_t-X'_s| \le \lba  \int_s^t\int_{|z|\le 2^{-n/\de}}  \hspace{-2mm} G(M_{\um},z) \,\wn (\rd u,\rd z) \rba + \lba  \int_s^t\int_{|z|\le 2^{-n/\de}}  \hspace{-2mm} G(M_{\um},z) \,\pi(\rd z)\rd u \rba .
\end{align*}
One uses \eqref{beta12eq1} and the estimate for $I_3$ in the last subsection to conclude that $H_{X'}(t)\ge 1/(\de+\be_M(t)+\e)$, as desired.

If $\be_M(t)\ge 1$. One has to use   $H_M(t)=H_X(t)$, Proposition \ref{jumpexp} and show \eqref{enough} for $X$. The desired lower bound is now less than one, hence it is enough to consider the increments of $X$.  Repeating the same lines as in the Section \ref{subsubsec}  entails the desired lower bound for $H_X(t)$. 

\subsection{Proof of Theorem \ref{exponent} (i)} \label{sectionproofjd}
When a non-degenerate Brownian integral exists, $M_t$ is the sum of $M_0$, $C_t$, $Y_t$, $X_t$ and the non-compensated Poisson integral in \eqref{jd}, recalling that $C_t = \int_0^t \sigma(M_{u-})\rd B_u$ and $Y_t$ is a Lebesgue integral. No regularity assumption is needed because any appearing drift (Lebesgue integral) is smoother than the Brownian integral which is only H\"older continuous.

\sk
We intend to show that almost surely, $\forall\, t\notin J$, $H_M(t) = \min(\frac{1}{\delta_t\beta_M(t)}, \frac{1}{2})$.  The following trivial fact is useful for our purpose. 
\begin{lemma}\label{trivial}
For any locally bounded $f,g:\rr\to\rr$ and $t\in\rr$,
\begin{align*}
H_{f+g}(t)\ge \min(H_f(t),H_g(t)) 
\end{align*}
where the equality occurs if $H_f(t)\neq H_g(t)$.  
\end{lemma}

\sk
Let $t\notin J$, then the non-compensated Poisson integral is locally constant around $t$.  As before, we distinguish two situations.
\begin{enumerate}
\item $\be_M(t)< 1$.  Then for any $s$ in a small neighborhood of $t$, 
\begin{align*}
M_s - M_t = (C_s - C_t) + (Y_s - Y_t) - (Y'_s-Y'_t) + (X'_s - X'_t),
\end{align*}
with $Y',X'$ defined in Section \ref{sec5.3}. Combining Proposition \ref{difexp},  Lemma \ref{techlemma}-(i) and Lemma \ref{trivial} yields that $H_{C+Y+Y'}(t) = 1/2$.  Meanwhile, $H_{X'}(t)=1/(\de_t\be_M(t))$ by \eqref{expX'}.  A further application of Lemma \ref{trivial} shows that $H_M(t) = \min(1/(\de_t\be_M(t)), 1/2)$ as soon as $1/2\neq 1/(\de_t\be_M(t))$. When they are equal,  the minimum  $1/(\de_t\be_M(t))$ is a lower bound for $H_M(t)$ by Lemma \ref{trivial}, it is also an upper bound by Proposition \ref{jumpexp}, which ends the proof.
\sk
\item $\be_M(t)\ge 1$. Still by Proposition \ref{difexp}, Lemma \ref{techlemma}-(i) and Lemma \ref{trivial}, one has $H_{C+Y}(t)=1/2$.  Observe that \eqref{enough} is proved when $\be_M(t)\ge 1$ (necessarily $\be_M(t)\de\ge 1$) without regularity assumption on $\wt G$.  So $H_X(t) = 1/(\de_t\be_M(t))$ follows by Proposition \ref{jumpexp}.  The rest of the proof repeats the arguments in the last paragraph. 
\end{enumerate}

\section{Computation of the pointwise multifractal spectrum} 
\label{secspec}%%%%%%
%%%%%%%%%%%%%%%%%%%%%%%%%%%%%%%%%%%%%%%%%%%%%%%%%%%%%%%%%%%%%%%%%%%%%%
%%%%%%%%%%%%%%%%%%%%%%%%%%%%%%%%%%%%%%%%%%%%%%%%%%%%%%%%%%%%%%%%%%%%%%
In this section, we compute the pointwise spectrum of  $M$ in all possible settings, i.e. Theorem \ref{mainwithdif} for jumps with diffusion ($\sigma\not\equiv 0$) and Theorems \ref{theo2}, \ref{mainresult}, \ref{mainresult2} for   jumps without diffusion ($\sigma\equiv 0$).  The main tool comes from geometric measure theory, the so-called ubiquity theorem, which consists in determining the Hausdorff dimension of some limsup sets. This theory   finds its  origin in Diophantine approximation and the localized version developed by Barral and Seuret \cite[Theorem 1.7]{barral2011localized} (see also \cite[Section 6]{barral2010increasing}) is very useful in studying random objects with varying pointwise spectra. Let us recall this theorem. 
 
 \sk
%%%%%%%%%%%%%%%%%%%%%%%%%%%%%%%%%%%%%%%%%%%%%%%%%%
%%%%%%% Localized ubiquity %%%%%%%%%%%%%%%%%%%%%%
%%%%%%%%%%%%%%%%%%%%%%%%%%%%%%%%%%%%%%%%%%%%%%%%%%
\begin{theorem}[\cite{barral2011localized, barral2010increasing}] \label{bs}
Let $\cS$ be a Poisson point process with intensity $\rd t\,\rd z/z^2$. Let  $I=(a,b)\subset\zu$ and $f:I \to [1,+\infty)$ be c\`adl\`ag whose set of jumps is denoted by $\mathfrak{C}$. 
Consider the sets
\begin{equation*}
S(I,f)  =   \left\{t\in I: \ \delta_t \geq f(t) \right\}
\;\;\hbox{ and }\;\;
\widetilde S(I,f)  =  \left\{t\in I: \ \delta_t  =  f(t) \right\},
\end{equation*}
where $\de_t$ is the approximation rate of $t$ by the point system $\cS$.  Almost surely for any $I=(a,b)\subset[0,1]$ and any c\`adl\`ag function $f : I\to \rr$,
\begin{equation*}
\dim_{\cH} S(I,f)= \dim_{\cH}\widetilde S(I,f)  =  \sup \{1/f(t): t\in I\backslash 
\mathfrak{C} \}.
\end{equation*}
\end{theorem}
%%%%%%%%%%%%%%%%%%%%%%%%%%%%%%%%%%%%%%%%%%%%%%%%

 %%%%%%%%%%%%%%%%%%%%%%%%%%%%%%%%%%%%%%%%%%%%%%%%%%%%%%%%%%%%%%%%%%%%%%
%%%%%%%%%%%%%%%%%%%%%%%%%%%%%%%%%%%%%%%%%%%%%%%%%%%%%%%%%%%%%%%%%%%%%%
\subsection{Proof of Theorem  \ref{mainwithdif}: Pointwise spectrum when $\sigma \not\equiv 0$}

\
  
When the Brownian integral does exist, the computation of the pointwise spectrum is easier relative to the Brownian integral vanishing case. Let $t\in(0,1)$ and  
  $$I^n_t :=\left( t-\frac{1}{n},t+\frac{1}{n}\right) \cap (0,1).$$

\begin{itemize}
\item If $h>1/2$, then   $D_M(t,h)=-\infty$ by item 1. of Theorem \ref{exponent}.
\item If $h<1/2$, then 
\begin{eqnarray}
E_M(h)\cap I^n_t &=& \llb s\in I^n_t : h= \frac{1}{\de_s\beta_M(s)}\wedge \frac{1}{2} \rrb \ala &=& \llb s\in I^n_t : h= \frac{1}{\de_s\beta_M(s)}\rrb = \llb s\in I^n_t : \de_s= \frac{1}{h\beta_M(s)}\rrb .\nonumber
\end{eqnarray}
But $\sup\{\be(x): x\in\rr\}<2$ so  that $\frac{1}{h\beta_M(s)}>1$ for any $s\in I^n_t$. This  yields that $\dim_{\cH} (E_M(h)\cap I^n_t) = \sup \llb h\beta_M(s) : s\in I^n_t \rrb$  by  Theorem \ref{bs}. Hence 
\begin{equation*} 
D_M(t,h) = \lim_{n\rightarrow +\infty} \dim_{\cH} (E_M(h)\cap I^n_t)  = h\cdot(\beta_M(t)\vee\beta_M(t-)).  
\end{equation*}
\item Consider finally $h=1/2$. For each $0\le h'<1/2$, set $\wt{E}_M({h'})= \llb s\in[0,1]: \de_s\ge \frac{1}{{h'}\beta_M(s)} \rrb $ which contains $E_M(h')$. By Theorem \ref{bs}, almost surely, 
\begin{equation} \label{equal} \dim_{\cH}  E_M({h'})  = \dim_{\cH} \wt{E}_M({h'})  \mbox{ for all } 0\le {h'}<1/2. \end{equation}
 Now decompose \begin{align*} I^n_t &= \( \bigcup_{{h'}<1/2} \( E_M({h'})\cap I^n_t \) \) \bigcup \( E_M(1/2)\cap I^n_t \) 
 %\ala &\subset \( \bigcup_{{h'}<1/2} \( \widetilde{E}_M({h'})\cap I^n_t \) \) \bigcup \( E_M(1/2)\cap I^n_t \) .
 \end{align*}
 Using the inclusion $E_M(h')\subset \wt E_M(h')$, the monotonicity of the sets $\{\wt E_M(h), 0\le h'<1/2\}$ and  \eqref{equal}, one has 
\begin{eqnarray}
1=\dim_{\cH}(I^n_t)&\le& 
%\( \dim_{\cH} \bigcup_{{h'}<1/2} \( \widetilde{E}_{h'}\cap I^n_t \) \) \vee   \( \dim_{\cH}  E_M({1/2}) \cap I^n_t   \) \ala
 \( \lim_{{h'}\uparrow 1/2} \dim_{\cH} \( \widetilde{E}_M({h'}) \cap I^n_t \) \) \vee\left(   \dim_{\cH}  E_M({1/2}) \cap I^n_t  \right) \ala
&=& \( \sup \llb \beta_M(s) : s\in I^n_t \rrb /2\) \vee  (\dim_{\cH}  E_M({1/2}) \cap I^n_t). \nonumber
\end{eqnarray} 
But $\sup\{\be(x): x\in\rr\}<2$ so that the above inequality shows $\dim_{\cH}  E_M({1/2}) \cap I^n_t =1$, which yields $D_M(t,1/2) = 1$.
\end{itemize}

%We refer the readers to \cite{BS} and \cite{BFJS} for the description of property $\mathfrak{P}$ and the fact that $\mathfrak{P}$ is realized almost surely by the Poisson system with intensity measure $dt\otimes\indiq_{z\in C(0,1)}dz/z^2$.
\subsection{Proof of Theorem \ref{theo2}: Linear parts of the pointwise spectrum when $\sigma\equiv 0$}

\

   We only prove the result for $t\in J$ and we treat separately three linear parts, the third being the constant $-\infty$ part of the spectrum.  The proof is simpler when $t$ is a continuous time for $M$, since in such case $\be_M(t)=\be_M(t-)$.   Set in the sequel
   $$\be_*(t) =  \min(\be_M(t), \be_M(t-)) \mbox{ and } \be^*(t) = \max(\be_M(t), \be_M(t-)).$$
  We also need the following notations  for a jump time $t$,
  $$I^*(t,n)= \begin{cases} \left(t,t+\frac{1}{n} \right), & \mbox{ if } \beta_M(t)= \be^*(t), \\
\left (t-\frac{1}{n},t\right ), & \mbox{ otherwise.}
\end{cases} $$
$$ I_*(t,n) = \begin{cases} \left(t,t+\frac{1}{n} \right), & \mbox{ if } \beta_M(t)=\be_*(t), \\
\left (t-\frac{1}{n},t\right ), & \mbox{ otherwise.}
\end{cases} $$
Clearly, $I^n_t = I^*(t,n)\cup \{t\}\cup I_*(t,n)$ and the union is disjoint. 

\begin{itemize}
\item If $h< 1/\be^*(t)$, there exists $\ep>0$ such that $h<1/({\beta^*(t)+\ep})$. But when $n$ is large enough, for any $ s\in I^n_t$, $\beta_M(s)<\beta^*(t)+\ep/2$ by the càdlàg property of the sample paths, which implies  
$$\frac{1}{h\beta_M(s)}>\frac{\beta^*(t)+\ep}{\beta^*(t)+\ep/2}>1$$
for all $s\in I_t^n$. Theorem \ref{bs} implies that 
\begin{eqnarray}
\dim_{\cH} E_M(h)\cap I^n_t  &=& \dim_{\cH} \llb s\in I^n_t : \de_s=\frac{1}{h\beta_M(s)} \rrb \ala &=& \sup \llb h\beta_M(s) : s\in I^n_t \rrb =h\cdot \sup\{\beta_M(s) : s\in I^n_t \}. \nonumber
\end{eqnarray}
for large $n$, which yields $D_M(t,h) = \ds{\lim_{n\rightarrow +\infty}} \sup \llb h\beta_M(s) : s\in I^n_t \rrb = h\cdot \beta^*(t)$.
\item If $h\in (1/\be^*(t), 1/\be_*(t))$, there exists  $\ep>0$ so that $h$ belongs to $ (1/(\be^*(t)-\e), 1/(\be_*(t)+\e))$. Let us consider separately $I^*(t,n)$ and $I_*(t,n)$. When $n$ is large enough, for all $s\in I^*(t,n)$, $$\frac{1}{h\beta_M(s)}\le \frac{\beta^*(t)-\ep}{\beta^*(t)-\ep/2}<1$$
by the càdlàg property of the sample paths. Hence $$E_M(h)\cap I^*(t,n) = \llb s\in I^*(t,n): \de_s=\frac{1}{h\beta_M(s)} \rrb$$ is empty, because $\de_s\ge 1$ uniformly a.s. due to Proposition \ref{cover01}.  When $n$ is large enough, for all $s\in I_*(t,n)$, $$\frac{1}{h\beta_M(s)} >  \frac{\beta_*(t)+\ep}{\beta_*(t)+\ep/2}>1$$ still by the c\`adl\`ag property of the sample paths.  Applying  Theorem \ref{bs} implies  that 
$$\dim_{\cH} E_M(h)\cap I^n_{t} = \dim_{\cH} E_M(h)\cap I_*(t,n)  = h\cdot \sup \llb \beta_M(s) : s\in I_*(t,n) \rrb$$
for large $n$.  Letting $n\to+\infty$ entails $D_M(t,h)  = h\cdot \beta_*(t)$.
\item If $h>1/{\beta_*(t)}$, there is  $\ep>0$ so that $h>1/({\beta_{*}(t)-\ep})$. But when $n$ is large enough,  for $s\in I^n_t$,  $\beta_M(s)>\beta_{*}(t)-\ep$, thus 
 $$h>\frac{1}{\beta_M(s)}\ge\frac{1}{\de_s\beta_M(s)} = H_M(s)$$
which yields $E_M(h)\cap I^n_t=\emptyset$. This proves  $D_M(t,h) = -\infty$.
\end{itemize}

%%%%%%%%%%%%%%%%%%%%%%%%%%%%%%%%%%%%%%%%%%%%%%%%%%%%%%%%%%%%%%%%%%%%%%
%%%%%%%%%%%%%%%%%%%%%%%%%%%%%%%%%%%%%%%%%%%%%%%%%%%%%%%%%%%%%%%%%%%%%%
\subsection{Statement of the general results for the pointwise spectrum when $\sigma= 0$}
As is said in the introduction, the absence of the Brownian integral reveals many problems at some extreme values of the pointwise spectrum.  Cases that have not been treated yet include
\begin{itemize}
\item $t$ is a continuous time,  $h=1/\be_M(t)$;
\item $t$ is a jump time, $h=1/\be^*(t)$ or $1/\be_*(t)$.
\end{itemize}

For a jump time $t$, the localization to a small neighborhood of $t$  makes essentially two different local behaviors appear,  one is captured by $\be^*(t)$, the other by $\be_*(t)$. When $1/h$ is different from both values,  only one of them is dominant,  as is observed in the proof for  the linear parts of the pointwise spectrum. However, both values will contribute to the computation of pointwise spectrum when $h$ is critical, i.e. $h= 1/\be_*(t)$ or $1/\be^*(t)$.  Further, $\de_t$ and local behaviors of the index process $\be_M$ contribute as well.  This is why we introduce the following notations.

For $t\in J$, we define  $\mathsf{b}^*  : I^*(t,n) \cup\{ t\} \rightarrow \rr$  by
\begin{equation} \mathsf{b}^* (s) =  
\begin{cases} \beta_M(s)& \mbox{ if } s\in I^*(t,n),\\
\beta_*(u)& \mbox{ if } s=t .\end{cases} \nonumber
\end{equation}
The map $ \mathsf{b}^*$ coincides with $\beta_M $ except at $t$ on its domain. 
Similarly, define $\mathsf{b}_*  : I_*(t,n) \cup\{ t\} \rightarrow \rr$ by   $$  \mathsf{b}_*( s)  =  
\begin{cases} \beta_M(s) & \mbox{ if } s\in I_*(t,n) ,\\
\be^*(t) & \mbox{ if } s=t .\end{cases}  $$  
 The maps $\mathsf{b}^*$, $\mathsf{b}_*$ depends clearly on $t$ and $n$, which are omitted for notational simplicity.
 
 We  write $t\in LM( f)$   to mean that $t$ is a strict local minimum for a mapping  $f$, i.e. $f(s) > f(t)$ for $s\neq t$ in a small neighborhood of $t$.

Finally, we introduce two functions $F_{\mbox{cont}}$ and $F_{\mbox{jump}}$ (see Figure \ref{figo2}) which correspond to different cases of the pointwise spectra. 
\begin{itemize}
\item  For a time $t$ where the process is continuous, we  will use
\begin{eqnarray*}
F_{\mbox{cont}} (    c,\gamma, h) =  \begin{cases} \gamma h & \ \mbox{ if } h\in \left[0, {1}/{\gamma}\right),\\    c & \ \mbox{ if } h  =1/\gamma,\\  -\infty& \ \mbox{ otherwise}.\end{cases}
\end{eqnarray*}
There will be only three possible values for $c$ (1, 0 and $-\infty$), that is, $F_{\mbox{cont}}$  is possible to be left continuous ($c=1$), right continuous ($c=-\infty$), or neither ($c=0$) on the discontinuous point $h=1/\be_M(t)$.
\item If  $t$ is a jump time for the process, we will use the function $F_{\mbox{jump}}$
\begin{eqnarray*}
F_{\mbox{jump}} ( c_1,c_2,\gamma_1,\gamma_2 , h ) =  \begin{cases} \gamma_1 \cdot h & \ \mbox{ if } \ds h\in \left [0,  {1}/{\gamma_1 } \right ),\\    c_1 & \ \mbox{ if } h \ds = {1}/{\gamma_1\ },\\ \gamma_2\cdot h & \ \mbox{ if } \ds h\in \left[{1}/{\gamma_1}, {1}/{\gamma_2}\right ),\\    c_2 & \ \mbox{ if } h \ds = {1}/{\gamma_2},\\   -\infty& \ \mbox{ otherwise},\end{cases}
\end{eqnarray*}
when $\gamma_1>\gamma_2$. There will be three possible values for $c_2$ (1, 0 and $-\infty$) and two for $c_1$ ($1$ and $\gamma_2/\gamma_1$).  Note that  $F_{\mbox{jump}}$ is either left continuous ($c_1=1$) or right continuous ($c_1=\ga_2/\ga_1$) on the first discontinuous point $h=1/\ga_1$,  and is possible to be left continuous ($c_2=1$), right continuous ($c_2=-\infty$) or neither ($c_2=0$) on the second discontinuous point $h= 1/\ga_2$.
\end{itemize} 

%
%where the process jumps, let $h_0= \frac{1}{\beta_M(t)\vee\beta_M(t)m}$ and $h_1= \frac{1}{\beta_M(t)\wedge\beta_M(t)m}$,  $F_{\mbox{jump}}(c_1,c_2,h)$ equals $(\beta_M(t)\vee\beta_M(t)m)h$ in the interval $[0,h_0)$, equals $(\beta_M(t)\wedge\beta_M(t)m)h$ in the open interval $(h_0,h_1)$,  
%\begin{equation}
%F_{\mbox{jump}}(c_1,c_2,h_0) =\begin{cases} 1, & \mbox{ if } c_1=\mbox{red (r)}, \\ \beta_M(t)\vee\beta_M(t)\cdot h_0, & \mbox{ if } c_1=\mbox{green (g)}. \end{cases} \nonumber
%\end{equation}
%\begin{equation}
%F_{\mbox{jump}}(c_1,c_2,h_1) =\begin{cases} 1, & \mbox{ if } c_2=\mbox{red (r)}, \\ 0, & \mbox{ if } c_2=\mbox{blue (b)},  \\ -\infty, & \mbox{ if } c_2 =\mbox{green (g)}. \end{cases} \nonumber
%\end{equation}
%and equals $-\infty$ elsewhere.

The several cases in the theorems below correspond to assigning a precise value to the discontinuous points of the pointwise spectrum, and various scenarii may occur, depending on the fact that $t$ is or not a strict local minimum for the processes $\be_M$, $\mathsf{b}^*$ and $\mathsf{b}_*$ (defined around $t$ on essentially disjoint domains).  The reader shall keep in mind the following heuristics:

\begin{center}
{\bf if $M$ is continuous at $t$, its pointwise spectrum looks like $F_{\mbox{cont}}$, \\if $t$ is a jump time, the pointwise spectrum looks like $F_{\mbox{jump}}$}. 
\end{center}

%%%%%%%%%%%%%%%%%%%%%%%%%%%%%%%%%%%%%%%%%%%%%%%%%%%%%%%%%%%%%%%%%%%%
%%%%%%%%%%%%%%%%%%%%%%%%%%%%%%%%%%%%%%%%%%%%%%%%%%%%%%%%%%%%%%%%%%%%%%%%%%%%%%%%%%
\begin{theorem}\label{mainresult} Assume that (H1)-(H5) hold with $\sigma \equiv 0$. 

\begin{enumerate}
\item Almost surely, for every $t\notin J$, 
the pointwise spectrum of $M$ at time $t$  is  given by 
\begin{eqnarray}\label{localspecM}
\hspace{-4mm}D_M(t,h)  &=& \begin{cases} F_{\mbox{cont}}( 1,\beta_M(t),h)  &  \mbox{ if } \  t\notin LM(\beta_M ), \\
F_{\mbox{cont}}(0,\beta_M(t),h)  &  \mbox{ if } \ t\in LM(\beta_M  ) \mbox{ and } \de_t=1, \\ F_{\mbox{cont}}(-\infty,\beta_M(t),h) & \mbox{ if } \ t\in LM(\beta_M  )\mbox{ and } \de_t\neq 1. \end{cases} \nonumber
\end{eqnarray}

\item \noindent       
 Almost surely, for all $t \in J$ and  $t\notin  LM(\mathsf{b}_*) \cup LM(\mathsf{b}^*) $ where $\mathsf{b}^*$ and $\mathsf{b}_*$ are defined locally around $t$, the pointwise spectrum at $t$ is 
   \begin{eqnarray*}
 D_M(t,h) =    F_{\mbox{jump}}(1, 1, \beta_*(t), \beta^*(t)    )  .
 \end{eqnarray*}
\end{enumerate}
\end{theorem}
%%%%%%%%%%%%%%%%%%%%%%%%%%%%%%%%%%%%%%%%%%%%%%%%%%%%%%%%%%%%%%%%%%%%
%%%%%%%%%%%%%%%%%%%%%%%%%%%%%%%%%%%%%%%%%%%%%%%%%%%%%%%%%%%%%%%%%%%%%%%%%%%%%%%%

This theorem covers the most frequent cases, i.e. when $t$ is a continuous time  or $t$ is a jump time and not a strict local minimum for  $\mathsf{b}^*$ and $\mathsf{b}_*$.

Next theorem covers all the "annoying" cases, i.e. when $t$ is a jump time and is a strict local minimum for at least one of the two functions   $\mathsf{b}^*$ and $\mathsf{b}_*$. Observe that this concerns at most a countable number of times.

%%%%%%%%%%%%%%%%%%%%%%%%%%%%%%%%%%%%%%%%%%%%%%%%%%%%%%%%%%%%%%%%%%%%
%%%%%%%%%%%%%%%%%%%%%%%%%%%%%%%%%%%%%%%%%%%%%%%%%%%%%%%%%%%%%%%%%%%%%%%%%%%%%%%%%%
\begin{theorem}\label{mainresult2} 
Assume that (H1)-(H5) hold with $\sigma \equiv 0$.  Almost surely, for any $t\in J$ that either belongs to $LM(\mathsf{b}^*)$ or $LM(\mathsf{b}_*)$, the following holds.
\begin{enumerate}
\item
If $t\notin LM(\mathsf{b}^*)$ and $t\in LM(\mathsf{b}_*)$, then 
\begin{eqnarray*}
  && \hspace{-14mm}  D_M(t,h) = \\
  &&  \hspace{-14mm} \ \begin{cases}  F_{\mbox{jump}}(1, 0,  \beta^*(t), \beta_*(t)  ,h) & \mbox{ if }  \Delta\beta_M(t)>0 \mbox{ and } \de_t =1, \\  F_{\mbox{jump}}(1, -\infty, \beta^*(t), \beta_*(t)  ,h) \!\!\!  & \mbox{ otherwise. }  % \Delta\beta_M(t)<0 \mbox{ or } \de_t\neq 1.
  \end{cases}  
 \end{eqnarray*}

\item
If $t\in LM(\mathsf{b}^*)$ and $t\notin LM(\mathsf{b}_*)$, then 
\begin{align*}
D_M(t,h) =  F_{\mbox{jump}}(\beta_*(t)/\beta^*(t) , 1,  \be_*(t), \be_*(t)   ,  h)
\end{align*}
\item If $t\in LM(\mathsf{b}^*)\cap LM(\mathsf{b}_*)$, then
\begin{eqnarray*}
 && \hspace{-17mm} \ D_M(t,h) = \\
 && \hspace{-21mm} \  
 \begin{cases}   
 F_{\mbox{jump}}(\beta_*(t)/\beta^*(t), 0,  \beta^*(t), \beta_*(t)  ,  h)  & \mbox{ if }    \Delta\beta_M(t)>0,\,  \de_t =1. \\   
 F_{\mbox{jump}}(\beta_*(t)/\beta^*(t) , -\infty,  \beta^*(t), \beta_*(t)   ,  h)  & \mbox{ otherwise. }  %\Delta\beta_M(t)<0 \mbox{ or } \de_t\neq 1.
\end{cases} 
\end{eqnarray*}

\end{enumerate}
\end{theorem}

When $t$ is a jump time, the behaviors of $M$ on the right hand-side and on the left hand-side of $t$ may differ a lot. So the pointwise spectrum  reflects the superposition of two local behaviors, which explains the formulas above. Though not easy to read, these formulas   are simple consequences of these complications that may arise as very special cases.

%%%%%%%%%%%%%%%%%%%%%%%%%%%%%%%%%%%%%%%%%%%%%%%%%%%%%%%%%%%%%%%%%%%%%%
%%%%%%%%%%%%%%%%%%%%%%%%%%%%%%%%%%%%%%%%%%%%%%%%%%%%%%%%%%%%%%%%%%%%%%

%%%%%%%%%%%%%%%%%%%%%%%%%%%%%%%%%%%%%%%%%%%%%%%%%%%%%%%%%%%%%%%%%%%%%%
%%%%%%%%%%%%%%%%%%%%%%%%%%%%%%%%%%%%%%%%%%%%%%%%%%%%%%%%%%%%%%%%%%%%%%
\subsection{Proof of Theorems \ref{mainresult} and \ref{mainresult2}}
\ 

Due to Theorem \ref{theo2}, it remains to prove the above theorems for the points of discontinuities of $F_{\mbox{cont}}$ and $F_{\mbox{jump}}$. We will only give the  proof for the discontinuous points of $F_{\mbox{jump}}$, the proof for  $F_{\mbox{cont}}$ can be written with some simplifications.   

%Let us start with an easy lemma.
%
%
%\begin{lemma}\label{lebesguetoutpoint}
%Assume that $\sigma\equiv 0$ and (H1)-(H6) hold. Almost surely, for Lebesgue-almost every $t\in[0,1]$, $h_M(t) =1/\beta_M(t)$.
%\end{lemma}
%\begin{proof}
%Using Theorem \ref{exponent},   one sees that a.s. for any $I=(a,b)\subset[0,1]$ and $\kappa\in[0,1]$,
% $$ \{t\in I : H_M(t) \le \kappa/\beta_M(t)\} = \{t\in I : \de_t \ge  1/\kappa \}$$
%which by Theorem \ref{bs} has Hausdorff dimension $\kappa$.  Next, one uses the fact that $\inf\{\de_t: t\in[0,1]\}\ge 1$ a.s. to decompose the interval $I$ as 
%\begin{equation}\label{lebesgue_eq2} I= \Big\{t\in I : H_M(t) = 1/\beta_M(t)\Big\} \ \bigcup  \ \( \bigcup_{n\ge 1} S_n \) \end{equation}
%where $S_n := \{t\in I : H_M(t)\le (1-1/n)/\beta_M(t) \}$. For every $n\ge 1$, the Lebesgue measure of $S_n$ is zero since it has Hausdorff dimension $1-1/n$. We deduce by \eqref{lebesgue_eq2} that for Lebesgue-a.e. $t\in I$, $H_M(t) = 1/\beta_M(t)$.
%\end{proof}

There are  two points of discontinuities for $F_{\mbox{jump}} $, which are  $1/{\beta^*(t)}$ and $1/{\beta_*(t)}$.

\subsubsection{First discontinuity $h=1/{\beta^*(t)}$}\label{seclm1}  We distinguish two cases.

\sk
\noindent{\bf Case 1 :}  $t\in LM(\mathsf{b}^*)$. This corresponds to items (2)-(3) in Theorem \ref{mainresult2}.  Then $\forall\, s\in I^*(t,n)$, $\frac{1}{h\beta_M(s)}=\frac{\beta^*(t)}{\beta_M(s)}<1$, which implies $E_M(h)\cap I^*(t,n)=\emptyset$. Notice that 
$$E_M(h)\cap I_*(t,n) = \llb s\in I_*(t,n) : \de_s = \frac{\beta^*(t)}{\beta_M(s)} \rrb.$$ 
For every $ s\in I_*(t,n)$ with $n$ large enough, one has $$\frac{\beta^*(t)}{\beta_M(s)}\ge \frac{\beta^*(t)}{\beta_*(t)+|\Delta\be_M(t)|/2} >1.$$
This  ensures that $\dim_{\cH}  E_M(h)\cap I_*(t,n)  = \sup \llb \frac{\beta_M(s)}{\beta^*(t)} : s\in I_*(t,n) \rrb$, still by Theorem \ref{bs}. Therefore, 
$$\dim_{\cH}  E_M(h)\cap I^n_t  = \dim_{\cH}  E_M(h)\cap (I_*(t,n)\cup \{t\})  = \dim_{\cH}  E_M(h)\cap I_*(t,n) ,$$
  which yields $D_M(t,h) = \ds{\lim_{n\rightarrow +\infty}} \dim_{\cH}  E_M(h)\cap I_*(t,n)  = \be_*(t)/\be^*(t).$

\sk
\sk
\noindent{\bf Case 2 : } $t\not\in LM(\mathsf{b}^*)$.  This is related to item (1) in Theorem \ref{mainresult2}. In this case, either $t$ is not a local minimum for $\mathsf{b}^*$, or $\mathsf{b}^*$ is locally constant (which happens with positive probability if $x\mapsto\be(x)$ has an interval of constancy).

 If $t$ is not a local minimum for $\mathsf{b}^*$, one can extract a   monotone  sequence $\{s_k\}\subset I^*(t,n)$ tending to $t$ such that 
\begin{equation}\label{sknotmin}
\be_M({s_k}) <\beta^*(t) \mbox{ and }  \lim_{k\to+\infty} \be_M(s_k) = \be^*(t).
\end{equation}
Since  $\be_M $ is càdlàg and the cardinality of $J$ is at most countable,  we can choose $s_k$  to be continuous times for $\be_M$. Let us first compute the pointwise spectrum of $M$ on times $s_k$ and deduce the result by a regularity restriction of the pointwise spectrum.
Fix $k\geq 1 $ and let $p$ be  large enough. For every $s\in I^p_{s_k}$,  one has $\frac{\beta^*(t)}{\beta_M(s)}>1$ by \eqref{sknotmin}. Further,  Theorem \ref{bs} ensures that $$\dim_{\cH} E_M(h)\cap I^p_{s_k}  = \sup \llb h\beta_M(s) : s\in I^p_{s_k} \rrb,$$  which yields that $D_M(s_k,h) = h\be_M({s_k})$.  Hence $$1\ge D_M(t,h) = \ds{\limsup_{s\rightarrow t}} \,D_M(s,h) \ge \ds{\limsup_{k\rightarrow +\infty}}\, D_M(s_k,h) = h\beta^*(t) = 1$$
where we used Lemma \ref{scs} in the first equality.  

If $\mathsf{b}^*$ is locally constant equal to $\be^*(t)$ in its domain,  then for $n$ large enough, one has $E_M(h)\cap I^*(t,n) = \{ s\in I^*(t,n) :  \de_s = \frac{1}{h\beta^*(t)} = 1\}$.  By Theorem \ref{bs} applied to the constant function $f(x)\equiv 1$, this set has Hausdorff dimension one, thus $ D_M(t,h) = 1 $.
%Applying Lemma \ref{lebesguetoutpoint}, one deduces that  $\mbox{Leb} (E_M(h) \cap I^*(t,n)) = \mbox{Leb} (I^*(t,n))$ where Leb denotes the Lebesgue measure,  so $\dim_\cH(E_M(h) \cap I^*(t,n))=1$.
%

\subsubsection{Second discontinuity $h=1/{\beta_*(t)}$}
 As before, we  distinguish two cases. 

\sk
\noindent{\bf Case 1 : } $t\in LM(\mathsf{b}_*)$. This is related to items (1) and (3) in Theorem \ref{mainresult2}.  For all $s\in I_*(t,n)$ with $n$ large enough,  one has $ \frac{\beta_*(t)}{\beta_M(s)}<1$, which implies that $E_M(h)\cap I_*(t,n)=\emptyset$. Notice that $E_M(h)\cap I^*(t,n) = \llb s\in I^*(t,n) : \de_s = \frac{\beta_*(t)}{\beta_M(s)} \rrb$ and that $\forall\, s\in I^*(t,n)$ with large $n$, $$\frac{\be_*(t)}{\beta_M(s)}< \frac{\be_*(t)}{\be_*(t)-|\Delta\be_M(t)|/2}<1.$$
One deduces $E_M(h)\cap I^*(t,n) = \emptyset$ for all $n$ large  enough. But
\begin{equation}
E_M(h)\cap \{t\} =\begin{cases} \{ t \} & \text{ if } \beta_M(t-)>\beta_M(t) \text{ and } \de_t = 1, \\ \emptyset & \text{ otherwise.} 
\end{cases} \nonumber
\end{equation}
Hence,  for all large $n$,
\begin{equation}
\dim_{\cH}  E_M(h)\cap I^n_t  = \dim_{\cH} E_M(h)\cap\{t\}  = \begin{cases}  \ \  0 & \mbox{ if } \beta_M(t-)>\beta_M(t) \mbox{ and } \de_t = 1, \\ -\infty & \text{ otherwise,} \end{cases} \nonumber 
\end{equation}
which  yields 
\begin{equation}
D_M(t,h) = \begin{cases}  \ \ 0 & \text{ if } \beta_M(t-)>\beta_M(t)  \text{ and } \de_t = 1, \\ -\infty &\text{ otherwise.} \end{cases} \nonumber
\end{equation}

\sk
\noindent{\bf Case 2 : } $t\not\in LM(\mathsf{b}_*)$. This corresponds to item (2) in Theorem \ref{mainresult2}. In such case, either $t$ is not a local minimum for $\mathsf{b}_*$, or $\mathsf{b}_*$ is locally constant in its domain around $t$. If $t$ is not a local minimum for $\mathsf{b}_*$. By a similar argument as in the second case in the Section \ref{seclm1}, we can prove that $D_M(s_k, h) = h\be_M({s_k})$ where $\{s_k\}\subset I_*(t,n)\setminus \cJ$ is a strictly monotone sequence tending to $t$ satisfying $\be_M({s_k})<\be_*(t)=\lim_{k\to\infty} \be_M(s_k)$. Therefore, Lemma \ref{scs} implies
\begin{equation}
1\ge D_M(t,h) = \limsup_{s\rightarrow t} D_M(s,h) \ge \limsup_{k\rightarrow +\infty} D_M(s_k,h) = h\be_*(t)=1. \nonumber
\end{equation}
If $\mathsf{b}_*$ is locally constant equal to $\be_*(t)$ in its domain around $t$, then for $n$ large enough,  $E_M(h)\cap I_*(t,n) = \{s\in I_*(t,n) : \de_s = \frac{\be_*(t)}{\beta_M(s)}= 1 \}$. An application of Theorem \ref{bs} to the constant function $f(x)\equiv 1$ yields 
%$\mbox{Leb} (E_M(h) \cap I_*(t,n)) = \mbox{Leb} (I_*(t,n))$, which yields 
$\dim_\cH(E_M(h)\cap I_*(t,n)) = 1$ for all $n$ large enough. Thus, $D_M(t,h)=1$.
 
%%%%%%%%%%%%%%%%%%%%%%%%%%%%%%%%%%%%%%%%%%%%%%%%%%%%%%%%%%%%%%%%%%%%%%
%%%%%%%%%%%%%%%%%%%%%%%%%%%%%%%%%%%%%%%%%%%%%%%%%%%%%%%%%%%%%%%%%%%%%%

\section{Examples}

\subsection{Variable order stable-like processes}\label{example1}
In 1988, R. Bass \cite{bass1988uniqueness} has shown the uniqueness in law  of a class of pure jump Markov process with generator
\begin{align*}
\cL^{\be}f(x) = \int_\rr (f(x+u) - f(x)- uf'(x)\indiq_{|u|\le 1} ) \be(x)|u|^{-1-\be(x)}\rd u
\end{align*}
under very weak conditions ($\be$ Dini-continuous and ranging in a compact set of $(0,2)$), that he called variable-order stable-like processes. It is well-defined for all $f\in C^2_c(\rr)$.  Denote by $\cF f(\xi) = \int e^{-ix\xi}f(x)\,\rd x$ the Fourier transform of $f$. One has for all $f\in C^{\infty}_c(\rr)$, 
\begin{align*}
\cF(\cL f(x))(\xi) = \int_{\rr_*} (e^{iu\xi} -1 - iu\xi\indiq_{|u|\le 1})\be(x)|u|^{-1-\be(x)}\rd u \cF f(\xi)
\end{align*}
Note that the right-hand side integral is the L\'evy-Khintchine representation of a certain symmetric $\be(x)$-stable distribution, so that $\cL^\be$ is a pseudo-differential operator with variable-order symbol that is close to $|\xi|^{\be(x)}$. By a variable change $u=\mbox{sign}(z)|z|^{1/\be(x)}$, the operator $\cL^\be$ is
\begin{align*}
\cL^\be f(x) = \int_\rr (f(x+\mbox{sign}(z)|z|^{1/\be(x)}) - f(x)- \mbox{sign}(z)|z|^{1/\be(x)}f'(x)\indiq_{|z|\ge 1} ) z^{-2}\rd z. 
\end{align*}
It is easy to verify by It\^o's formula that any solution $M$ to \eqref{jd} with coefficients 
\begin{equation}\label{coef}
 \sigma=b=0 \ \ \mbox{ and } \ \  G(x,z)=F(x,z)=\mbox{sign}(z)|z|^{1/\be(x)}
\end{equation}
 solves the martingale problem associated with $\cL^\be$ on $C^2_c(\rr)$.  Let $\be$ be Lipschitz continuous that ranges in a compact set of $(0,2)$,  the conditions (H1)-(H4) are satisfied so that the  SDE \eqref{jd} has a unique strong solution which is a strong Markov process.   By the uniqueness of the martingale problem (\cite[Theorem 2.2]{bass1988uniqueness}), $M$ is a variable-order stable-like process associated with $\be$. Observe that one has $b=\wt G = 0$ so that the condition (H5) is automatically satisfied without further regularity  assumption on $\be$. Hence, our results hold for the class of variable order stable-like processes with Lipschitz continuous $\be$ that ranges in a compact set of $(0,2)$.

\subsection{SDE driven by stable L\'evy processes}
\label{example2}

Recently, there has been much interest in SDE driven by stable L\'evy processes, see \cite{bass2003stableSDE,fu2010positiveSDE, li2011strongsolution, fournier2013pathuniquenessofstableSDE}. Our results can be applied to deduce multifractal nature of non-degenerate stable-driven SDEs. 

Recall that any symmetric $\al$-stable L\'evy process can be written as
\begin{align*}
Z_t = 
\begin{cases}
\int_0^t\int_{\rr_*} \mbox{sign}(z)|z|^{1/\al} \wt N(\rd s,\rd z) & \mbox{ if } 1<\al<2 \\
\int_0^t\int_{\rr_*} \mbox{sign}(z)|z|^{1/\al} N(\rd s,\rd z) & \mbox{ if } 0<\al<1
\end{cases}
\end{align*}
where $N$ is a Poisson measure whose intensity  $c\rd t\,\rd z/z^2$ for some positive finite constant $c$. Assume without loss of generality that $c=1$.
Note that the stable driven SDE
\begin{align*}
X_t = X_0 + \int_0^t g(X_{s-}) \,\rd Z_s + \int_0^{t} b(X_s)\,\rd s
\end{align*}
can be written as
\begin{align*}
X_t = \begin{cases}
x + \int_0^t g(X_{s-})\mbox{sign}(z)|z|^{1/\al} \,\wt N(\rd s, \rd z) + \int_0^{t} b(X_s)\,\rd s & \mbox{ if } 1<\al<2, \\
x + \int_0^t g(X_{s-})\mbox{sign}(z)|z|^{1/\al} \, N(\rd s, \rd z) + \int_0^{t} b(X_s)\,\rd s & \mbox{ if } 0<\al<1.
\end{cases}
\end{align*}
Let $g:\rr\to\rr$ be bounded above and bounded below away from zero in absolute value.  Let $b$ and $g$ be sufficiently smooth such that (H5) holds. Then the conditions (H1)-(H5) are satisfied with $G(x,z)=g(x)\mbox{sign}(z)|z|^{1/\al}$ and $\be(x)=\al$. Our results establishes that $X$ is homogeneously multifractal and its spectrum is the same as the driving process $Z$.   It is possible to extend the result to $g$ unbounded using some localization argument.  
%
% We consider the pure jump equations, i.e. $\sigma =b \equiv 0$. Let $g: \rr\to \rr$ be smooth and bounded below from $0$,  $\widetilde\be : \rr \to \rr$ be smooth with its range included in some compact subset of $(0,2)$. We take $$G(x,z)=g(x)z^{\frac{1}{\widetilde\be(x)}}.$$  If $\widetilde\be \equiv \al \in (0,2)$, $M$ is the solution to the Lévy stable-driven SDE 
%\begin{equation}\label{stabledrivenSDE} M_t= \int_0^t\int_{C(0,1)} g(M_{s-}) dL^\al_s \end{equation}
% with $L^\al$ an $\al$-stable Lévy process. Indeed, we remark that the image measure of $dz/z^2$ by the mapping $z\mapsto z^{1/\al}$ is $const\cdot dz/z^{1+\al}$.  Our main theorem establishes that \eqref{stabledrivenSDE} has the same multifractal nature as an $\al$-stable Lévy process, under the condition that $g$ is also bounded from above. By a classical localization argument (see for instance Proposition 2.4 in \cite{fu2010positiveSDE}), one can easily extend the result to unbounded coefficient $g$.  If $g$ is constant, then $M$ is a stable-like process in the sense of Bass \cite{bass1988uniqueness}. We clarifiy this point in Section \ref{sectionrelationstablelike}. 
%
% It is worth mentioning that the intensity measure can be any stable Lévy measure, i.e. $|z|^{-1-\al}dz$ with $\al \in(0,2)$ (in this case, the definition of the set  $\cG$ must be adapted). With a suitable change of measure for the Poisson integral, it suffices to consider our specific Lévy measure $\pi(dz)=dz/z^2$. 

\appendix
\section{Auxiliary results}
\begin{preuve}{\it \ of Proposition \ref{strongsolution} : }
%It remains to check the local Lipschitz condition for $G$. 
%\sk
%\noindent(i) We check the growth condition.  We divide the integral into two parts, use \eqref{Gzsm} and the uniform boundedness of $G$ to get
%\begin{eqnarray}
%\int_{C(0,1)} G(x,z)^2\,\frac{dz}{z^2} &=& \int_{C(0,z_0)} G(x,z)^2\,\frac{dz}{z^2} + \int_{C(z_0,1)} G(x,z)^2\,\frac{dz}{z^2}. \ala
%&\le& \int_{C(0,z_0)} |z|^{1+2\ep_0}\,\frac{dz}{z^2} + \int_{C(z_0,1)} \,\frac{dz}{z^2} \ala
%&=&  \frac{2z_0^{2\ep_0}}{2\ep_0} + 2 \( \frac{1}{z_0}-1 \) :=K_0 \le K_0(1+x^2). \nonumber
%\end{eqnarray}
For each $m\in\nn^*$, consider distinct $|x|,|y|\le m$. Assume without loss of generality that $\ln |G(x,z)| > \ln |G(y,z)|$, then there is a finite constant $c_m$ so that 
\begin{align}\label{lipofG}
|G(x,z) - G(y,z)| &=  |G(x,z)| \( 1- e^{\ln |G(y,z)| - \ln |G(x,z)|}\) \nonumber \\ \nonumber
&\le |G(x,z)| (\ln |G(x,z)| - \ln |G(y,z)|) \\
&\le  c_m |x-y| |G(x,z)| (\ln |z|) 
\end{align}
where we used the inequality $1-e^{-u}\le u$ for all $u>0$ and our local Lipschitz condition on $G$. To conclude, it remains to show that 
\begin{align*}
 \int_{|z|\le 1}  |G(x,z)|^2 (\ln |z|)^2 \,\pi(\rd z)
\end{align*}
 is bounded above by a finite constant, uniformly for $|x|\le m$. 

Let $\e_0 = 2- \sup\{\be(x): x\in\rr\}$ and $0<\e<2/(2-\e_0/2) -1$.  By the asymptotically stable-like assumption (H4), $\e_0>0$ and there is a $r(\e_0)>0$ so that 
\begin{align*}
\int_{|z|\le r(\e_0)} |G(x,z)|^2 (\ln |z|)^2 \,\pi(\rd z) &\le \int_{|z|\le r(\e_0)} |z|^{2/(\be(x)+\e_0/2) -2} (\ln |z|)^2 \,\rd z \\
&\le \int_{|z|\le r(\e_0)} |z|^{\e -1} (\ln |z|)^2 \,\rd z
\end{align*}
Replace $r(\e_0)$ by a smaller number if necessary to ensure $|z|^{\e/2}(\ln|z|)^2\le 1$ for all $|z|\le r(\e_0)$,  then this integral is bounded above by  $(4/\e)r(\e_0)^{\e/2}$.  On the other hand, the linear growth condition (H1) implies that  for some universal $K$,  each $x\in\rr$ satisfies
\begin{align*}
\int_{r(\e_0)<|z|\le 1} |G(x,z)|^2 (\ln |z|)^2 \,\pi(\rd z) \le K(\ln|r(\e_0)|)^2(1+x^2)
\end{align*}
which ends the proof. 
\end{preuve}

\sk
\begin{preuve}{\it \ of Lemma \ref{techlemma} :}
The first item is obvious.  For the second item, there is nothing to prove if $H_g(x_0)=0$. Suppose $k< H_g(x_0)\le k+1$ for some $k\in\nn$, in particular $g$ is continuous on $x_0$. By assumption, $f\in C^{k+1}(\rr)$. Hence there is a polynomial $P_f$ of degree at most $k$ so that 
$$|f(y)- f(y_0) - P_f(y-y_0)| = O(|y-y_0|^{k+1} ) $$
as $y$ tends to $y_0=g(x_0)$.  Denote by $P_g$ the polynomial of degree at most $k$ such that 
$$|g(x)- g(x_0) - P_g(x-x_0)| = o(|x-x_0|^{H_g(x_0)-\e})$$
as $x$ tends to $x_0$, for any fixed $\e>0$. Write 
\begin{align*}
P_f\circ P_g(h) = P(h) + P'(h)
\end{align*}
where $P$ is a polynomial of degree $k$ and $P'$ is a polynomial of order $O(h^{k+1})$ as $h\to 0$.  One has 
\begin{align*}
|f(g(x)) - f(g(x_0)) - P(x-x_0)| &\le |f(g(x))-f(g(x_0)) - P_f(g(x)-g(x_0))| \\
&+ |P_f(g(x)-g(x_0))- P_f(P_g(x-x_0))| \\
&+ |P_f(P_g(x-x_0)) - P(x-x_0)| \\
&= O(|g(x)-g(x_0)|^{k+1}) + o(|x-x_0|^{H_g(x_0)-\e}) + O(|x-x_0|^{k+1})
\end{align*}
as $|x-x_0|\to 0$.  The first term is of order $O(|x-x_0|^{k+1})$, thus the above sum is of order $o(|x-x_0|^{H_g(x_0)-\e})$. Letting $\e\to 0$ ends the proof. 
\end{preuve}

\sk
\begin{preuve}{\it \ of Lemma \ref{techlemma2} : } 
Let $x\in D_m$.  By (H4), there is a constant $r_m>0$ such that 
\begin{align*}
\int_{|z|\le r_m} |G(x,z)|\,\pi(\rd z) \le \int_{|z|\le r_m} |z|^{\frac{1}{\be(x)+1/(2m)} -2} \,\rd z \le \int_{|z|\le r_m} |z|^{\frac{1}{1-1/2m} -2} \,\rd z := c^1_m<+\infty.
\end{align*}  
Applying Cauchy-Schwartz inequality and the growth condition (H1) implies  
\begin{align*}
\int_{r_m<|z|\le 1} |G(x,z)|\,\pi(\rd z) \le \pi(\{z: r_m<|z|\le 1\})^{1/2} \times K(1+m^2)^{1/2} := c^{2}_m < +\infty.
\end{align*}
The first property is proved.  

Now we turn to the $L^1$ local Lipschitz condition for $G$ on the domain $D_m$.  Let $x,y\in D_m$. In the light of \eqref{lipofG}, it suffices to prove that the integral
\begin{align*}
\int_{|z|\le 1} |G(x,z)| \ln(1/z) \,\pi(\rd z)
\end{align*} 
is bounded above by a constant, uniformly for all $x\in D_m$. That $x$ belongs to $D_m$ and (H4) makes the integral over $\{z: |z|\le r_m\}$ finite uniformly for $x\in D_m$. On the other hand, one uses Cauchy-Schwartz inequality to bound above the integral over the set $\{z: r_m<|z|\le 1\}$, which has finite $\pi$-measure.
\end{preuve}

\section{Pointwise exponent of the Brownian integral}

Recall the martingale representation theorem.
%%%%%%%%%%%%%%%%%%%%%%%%%%%%%%%%%%%%%%%%%%%%%%%%%%%%%%%%%%%%%%%%%%%%%%
\begin{theorem*}[Dambis-Dubins-Swartz,   Th. 5.1.6  \cite{revuzyor1999}] 
Let $\cM$ be a $(\cF_t,\mathbb{P})$-continuous local martingale   such that a.s. $\cM_0=0$ and $\langle \cM \rangle_{+\infty} = +\infty$. Let 
$$T_t=\inf\{s\ge 0 : \langle \cM \rangle_s > t \},$$
then $B_t := \cM_{T_t}$ is a $(\cF_{T_t})$-Brownian motion and a.s.  $\forall\, t\in\rr^+,\, \cM_t = B_{\langle M\rangle_t}$.

%%%%%%%%%%%%%%%%%%%%%%%%%%%%%%%%%% 
%below is another case where the enlargement of probability space is necessary
%%%%%%%%%%%%%%%%%%%%%%%%%%%%%%%%%%
%If $<M>_{\infty}<\infty$ with positive probability. There exist an enlargement $(\widetilde{\Omega},\widetilde{\cF_t}, \widetilde{\mathbb{P}} )$ of $(\Omega,\cF_{T_t}, \mathbb{P} )$ and a Brownian motion $\tilde \beta$ on $\widetilde{\Omega}$ independent of $M$ such that the process
%$$B_t = M_{T_t} + \int_0^t \indiq_{s> <M>_{\infty}} d\tilde \beta_s$$
%is a Brownian motion. As in the first point, one has $M_t = B_{<M>_t}, \forall t\ge 0$.
\end{theorem*}
%%%%%%%%%%%%%%%%%%%%%%%%%%%%%%%%%%%%%%%%%%%%%%%%%%%%%%%%%%%%%%%%%%%%%%

\begin{preuve}{\it \ of Proposition \ref{difexp} :} 
Recall that $\calX_t=\int_0^t \sigma(M_s)\,dB_s$ is a local martingale starting from $0$. The quadratic variation process of $\calX$  $$\langle \calX \rangle_t = \int_0^t \sigma(M_s)^2 \,ds,$$
  satisfies $\langle \calX\rangle_{\infty}=\infty$ almost surely, since $\sigma$ stays away from $0$ by assumption. Applying Theorem of Dambis-Dubins-Swartz to $\calX$, one can find a standard Brownian motion $\widetilde{B}$ on $(\cF, \pp)$ such that a.s. $\forall\, t,\, \calX_t = \widetilde{B}_{\langle \calX \rangle_t}$. 

First computation yields a.s. for every $t\in\rr^+$, $\forall\,r>0$, for all $u\in B(t,r)$,
 \begin{eqnarray}\label{bilip1}
c|u-t| \le \lba \langle \calX \rangle_u - \langle \calX \rangle_t \rba &\le& \lba \int_t^u C (1+|M_s|)^2 ds \rba 
\le    C |u-t| \label{eqsigma}, %\sup_{s\in B(t,r)} (1+|M_s|^2) 
\end{eqnarray}
 where we used   that $\sigma$ stays away from $0$ to find the constants $c,\,C\in\rr^+_*$.

By L\'evy's modulus of continuity for Brownian motion (Theorem 1.2.7 of \cite{revuzyor1999}), for every $ \ep>0$, a.s. for every $t$, for $u$ sufficiently close to $t$,  one has by \eqref{eqsigma} \begin{eqnarray}
|\calX_u-\calX_t| = |\widetilde{B}_{\langle \calX \rangle_u} - \widetilde{B}_{\langle \calX \rangle_t}|
\le C' |\langle \calX \rangle_u - \langle \calX \rangle_t|^{\frac{1}{2} - \ep} 
\le C' |u-t|^{\frac{1}{2} - \ep}.\nonumber
\end{eqnarray}
 Hence, almost surely, $\forall\,t,\, H_{\calX}(t) \ge \frac{1}{2} - \ep$. 

On the other hand, Dvoretzky \cite{dvoretzky1963irregularityofBM} proved that, for a standard Brownian motion $B$, there exists a constant $K>0$, such that almost surely $$\forall\,t,\quad \limsup_{h\to 0^+} \frac{|B_{t+h}-B_t|}{h^{1/2}}\ge K.$$
Applying Dvoretzky's Theorem to our Brownian motion $\widetilde{B}$, we get  that almost surely for every $t\ge 0$, there exists a positive sequence $(h_n)_{n\geq 1}$ converging to zero  such that 
\begin{eqnarray}\label{bilip2}
|\widetilde{B}_{\langle \calX \rangle_t + h_n}-\widetilde{B}_{\langle \calX \rangle_t}|\ge K|h_n|^{1/2}.
\end{eqnarray}
As $t\mapsto \langle \calX \rangle_t$ is  a strictly increasing (always by the assumption that $\sigma$ stays away from $0$) continuous function, there exists a sequence $(u_n)_{n\geq 1}$ such that $\langle \calX \rangle_t + h_n = \langle \calX \rangle_{u_n}$. By the first inequality of \eqref{bilip1}, \begin{equation}\label{bilip3}
|h_n|\ge c|u_n-t|.
\end{equation}
It follows form \eqref{bilip2} and  \eqref{bilip3} that 
\begin{align}
|\calX_{u_n}-\calX_t| &= |\widetilde{B}_{\langle \calX\rangle_{u_n}} - \widetilde{B}_{\langle \calX\rangle_t}| = |\widetilde{B}_{\langle \calX \rangle_t + h_n} - \widetilde{B}_{\langle \calX \rangle_t}|  \ge Kc|u_n-t|^{1/2} \nonumber,
\end{align}
  This yields a.s.  $\forall\,t,\,H_\calX(t) \le 1/2$, and letting $\ep$ tend to $0$ gives the result.
\end{preuve}

\section{Existence of tangent processes}%%%%%%%%%%%%%%%%%%%%%%%%%%%%%
%%%%%%%%%%%%%%%%%%%%%%%%%%%%%%%%%%%%%%%%%%%%%%%%%%%%%%%%%%%%%%%%%%%%%%
%%%%%%%%%%%%%%%%%%%%%%%%%%%%%%%%%%%%%%%%%%%%%%%%%%%%%%%%%%%%%%%%%%%%%%

In order to describe the local structure of stochastic processes which are often rough (not differentiable), several authors consider the tangent processes associated with them, see for instance \cite{falconer2003tangentprocesses}. Precisely, given a stochastic process $X$ and $t_0$ a fixed time, one wonders if there exist two sequences $(\al_n)_{n\geq 1}$, $(r_n)_{n\geq 1}$ decreasing to zero such that the sequence of process $( r_n (X_{t_0+\al_nt} - X_{t_0}) ) _{t\ge 0}$ converges in law to some limit process $(Y_t)_{t\ge 0}$, and call it, if exists, a tangent process.  One observes in Theorem \ref{mainwithdif} and Theorem \ref{mainresult} that the pointwise spectrum of the process $M$ looks like (but not exactly) the spectrum of some L\'evy process. Then natural questions concern the connections between the pointwise spectrum of the process at $t_0$ and its tangent process  at this point. In the stable-like case, we show the existence of tangent processes of $M$, which are some stable Lévy processes.   
Their spectra coincide with the pointwise  spectra of $M$ at time $t$ except for one value of $h$. Here, the scaling $(r_n, \al_n$) must be carefully chosen and  plays an important role.

Throughout this section, the Skorokhod space of c\`adl\`ag functions on $[0,1]$ is endowed with the uniform convergence topology. 
We consider the function $G_0  (x,z) = \mbox{sign}(z)|z|^{1/\be(x)}$ with $\be$ Lipschitz continuous and $\overline{\mbox{ Range } \be} \subset (0,2)$, and the pure jump diffusion still denoted by $M$:
\begin{align*}
M_t = \int_0^t\int_{|z|\le 1} G_0(M_{\sm},z)\,\wn(\rd s, \rd z).
\end{align*}

\begin{proposition}\label{tangentexample} Let $t_0\ge 0$ be fixed, conditionally on $\cF_{t_0}$, the family of processes $\( \frac{M_{t_0+\al t}-M_{t_0}}{\al^{1/\be({t_0})}}\) _{t\in[0,1]}$ converges in law to a stable L\'evy process with L\'evy measure $\be_M({t_0})u^{-1-\be_M( {t_0})}\,du$, when $\al\to 0$.
\end{proposition}

The next lemma gives some moment estimate for $M$ near $0$. The second point was proved in \cite{barral2010increasing}, we still prove it for the sake of completeness. Let us introduce the stopping times  for every $\eta>0$
\begin{equation}
\tau_\eta := \inf \{ t>0 : \beta_M(t)> \be(0) + \eta \} \nonumber.
\end{equation}

%%%%%%%%%%%%%%%%%%%%%%%%%%%%%%%%%%%%%%
%% Moment estimates %%%%%%%%%%%%%%%%%%
%%%%%%%%%%%%%%%%%%%%%%%%%%%%%%%%%%%%%%

\begin{lemma} \label{momentofjump0} Let $\eta>0$ be small.  
\begin{itemize}
\item[(i)] If $\beta(0)\ge 1$, for every $\gamma\in(\beta(0)+\eta,2)$, there exists a constant $c_\gamma$ such that $\forall\, \al>0$, 
\begin{equation}
\E[|M_{\al\wedge\tau_\eta}|^\gamma] \le c_\gamma \al. \nonumber
\end{equation}
\item[(ii)] If $\beta(0)<1$, for every $\gamma\in(\beta(0)+\eta,1\wedge 2\beta(0))$, the same moment inequality holds. 
\end{itemize}
\end{lemma}

\begin{proof}
(i) Since $M$ is a  martingale, by Burkholder-Davis-Gundy inequality and subadditivity, one has 
\begin{eqnarray}
\E[|M_{\al\wedge\tau_\eta}|^\gamma] &\le& \E[\sup_{0\le t\le \al\wedge\tau_\eta} |M_t|^\gamma] \ala 
&\le& c_\gamma \E \[ \lba \int_0^{\al\wedge\tau_\eta} \int_{|z|\le 1} |G_0(M_{s-},z)|^2 N(dsdz)\rba ^{\gamma/2} \] \ala
&\le& c_\gamma \E \[ \int_0^{\al\wedge\tau_\eta} \int_0^1 |G_0(M_{s-},z)|^\gamma \,N(dsdz) \] \ala
&=& c_\gamma \E \[ \int_0^{\al\wedge\tau_\eta} \int_0^1 |G_0(M_{s-},z)|^\gamma \,dz/z^2ds \] \nonumber
\end{eqnarray} 
For every $s\in[0,\tau_\eta)$, one has
\begin{eqnarray}
\ \int_0^1 \hspace{-2mm}|G_0(M_{s-},z)|^\gamma dz/z^2 \!=\! \int_0^1 \hspace{-2mm}|z|^{\gamma/\beta_M(s-)} dz/z^2  \!\le\! \int_0^1 |z|^{\gamma/(\beta(0)+ \eta)} dz/z^2 \!<\! +\infty, \nonumber 
\end{eqnarray}
where we used that $\gamma> \beta(0)+\eta$. Hence, $$\E[|M_{\al\wedge\tau_\eta}|^\gamma] \le c_\gamma \E\[ \int_0^{\al\wedge\tau_\eta}\,ds\] \le c_\gamma\al.$$

(ii) For every $s\in[0,\tau_\eta)$ with $\eta$ small enough, it makes sense to separate the compensated Poisson integral, see Lemma \ref{fivar}. Using $(a+b)^\gamma\le (a^\gamma+b^\gamma)$ for all $(a,b)\in\rr^2_+$, $\ga\le 1$ and integral type subadditivity, one has
\begin{eqnarray}
\E[|M_{\al\wedge\tau_\eta}|^\gamma] &\le& \E \[ \lba\int_0^{\al\wedge\tau_\eta} \int_{|z|\le } |G_0(M_{s-},z)| N(dsdz)\rba ^{\gamma} \] \ala
&& \ \ \ \ \ + \,\E \[ \lba\int_0^{\al\wedge\tau_\eta} \int_{|z|\le 1} |G_0(M_{s-},z)| \,dz/z^2ds \rba ^{\gamma} \] \ala
&\le& c_\gamma \E \[ \int_0^{\al\wedge\tau_\eta} \int_0^1 |G_0(M_{s-},z)|^\gamma \,dz/z^2ds \] . \nonumber
\end{eqnarray}
Repeating the arguments of the first point yields the result.
\end{proof}

\begin{lemma}\label{nonss0} Let $x_0$ be fixed. For all $\gamma>\be(x_0)$, there exist strictly positive constants $C_{\gamma}$ and $\delta$ such that  for all $  x\in B(x_0,\de)$
$$\int_{|z|\le 1}|G_0(x,z)-G_0(x_0,z)|^{\gamma}\,\pi(dz) \le C_{\gamma}|x-x_0|^\gamma.$$
\end{lemma}
It is easy to check Lemma \ref{nonss0}. Now we prove Proposition \ref{tangentexample}, using the self-similarity of the limit process and last two lemmas.

\sk
\begin{proof} 
By  the Markov property, it is enough to prove the proposition for $t_0=0$. Let us introduce  
$$\cL_t=\int_0^t\int_{|z|\le 1}G_0(0,z)\wn(\rd s, \rd z), \ \ \mathcal{S}_t = \int_0^t\int_{\rr^*} G_0(0,z) \wn(\rd s, \rd z).$$
Note that $\cL$ and $\mathcal{S}$ are pure jump Lévy processes whose Lévy measure are $\beta(0)|z|^{-\beta(0)-1}\indiq_{|z|\le 1}dz$ and $\beta(0)|z|^{-\beta(0)-1}dz$, respectively. An application of It\^o formula shows that $\mathcal{S}$ is a  symmetric $\beta(0)$-stable L\'evy process, thus is $1/\beta(0)$-self-similar, meaning that for every $ \al>0$,
 \begin{align*}
 \( \al^{-1/\beta(0)}\mathcal{S}_{\al t} \) _{t\in[0,1]} = \( \mathcal{S}_t \) _{t\in[0,1]}
 \end{align*}
in law, see for instance Chapter 3 of Sato \cite{sato2013}. Observe that  $\forall\, \de>0$,
\begin{eqnarray}
\pp\( \sup_{0\le t\le 1} \lba \al^{-1/\beta(0)}(\calL_{\al t} - \mathcal{S}_{\al t}) \rba \le\de \) \ge \pp \( N([0,\al]\times \{z:|z|>1\})=0 \) =e^{-\al}  \to_{\al\downarrow 0} 1 \nonumber.
\end{eqnarray}
This computation yields that $$\al^{-1/\beta(0)}\sup_{t\in[0,1]}|\calL_{\al t}-\mathcal{S}_{\al t}| \to 0$$
in probability, when $\al\to 0$. Recall that the self-similarity of $\mathcal{S}$ ensures that  $(\al^{-1/\beta(0)}\mathcal{S}_{\al t})_{t\in[0,1]}$ converges (equals) in law to $(\mathcal{S}_{t})_{t\in[0,1]}$, thus the process $(\al^{-1/\beta(0)}\calL_{\al t})_{t\in[0,1]}$ converges in law to $(\mathcal{S}_t)_{t\in[0,1]}$. To conclude, it remains to prove the following
\begin{eqnarray}
\al^{-1/\beta(0)}\Delta_\al \to 0 \mbox{ in probability}, \nonumber
\end{eqnarray}
where $\Delta_\al := \sup_{0\le t\le \al} |M_t-\calL_t|$.
There are two cases.

\medskip
\noindent{\bf Case 1 : $\beta(0)\ge 1$.} Applying the Burkholder–Davis-Gundy inequality and a subadditivity property, one has, for every $\gamma\in(\beta(0)+\eta,2)$,
\begin{eqnarray}
\E[|\Delta_{\al\wedge\tau_\eta}|^\gamma] &\le& c_\gamma \E\[ \int_0^{\al\wedge\tau_\eta}\int_{C(0,1)} |G_0(M_{s-},z)-G_0(0,z)|^\gamma \,dz/z^2ds \] \ala
&\le& c_\gamma \E\[ \int_0^{\al\wedge\tau_\eta} |M_s|^\gamma ds \] \le c_\gamma \int_0^{\al} \E[|M_{s\wedge\tau_\eta}|^\gamma] \,ds \le c_\gamma \al^2,\nonumber
\end{eqnarray}
where we used Lemma \ref{nonss0} and Lemma \ref{momentofjump0}. Hence, for every $ \de>0$, one has
\begin{eqnarray}\label{cheby*}
\pp \( \al^{-1/\beta(0)}\Delta_\al \ge \de \) \le \pp \( \tau_\eta \le \al \) + \pp \( \al^{-1/\beta(0)}\Delta_{\al\wedge\tau_\eta} \ge \de \) ,
\end{eqnarray}
where 
$\lim_{\al\downarrow 0+} \pp (\tau_\eta\le \al) = \pp (\tau_\eta = 0) =0$ and
\begin{equation}\label{cheby2*}
\pp \( \al^{-1/\beta(0)}\Delta_{\al\wedge\tau_\eta} \ge \de \) \le \de^{-\gamma}\al^{-\gamma/\beta(0)} \E[|\Delta_{\al\wedge\tau_\eta}|^\gamma] \le c_{\de,\gamma} \al^{2-\gamma/\beta(0)} \to 0,
\end{equation} since $2\beta(0)\ge 2> \gamma > \beta(0)+\eta$.

\sk
\noindent{\bf Case 2 : $\beta(0)<1$. } As in Lemma \ref{momentofjump0}, for every $s\in[0,\tau_\eta)$ with $\eta$ small enough, it makes sense to separate the compensated Poisson measure. By subadditivity, for every $\gamma\in(\beta(0)+\eta, 1\wedge 2\beta(0))$,
\begin{eqnarray}
\E[|\Delta_{\al\wedge\tau_\eta}|^\gamma] &\le& c_\gamma \E\[ \lba \int_0^{\al\wedge\tau_\eta} \int_{C(0,1)} |G_0(M_{s-},z)-G_0(0,z)| \,N(dsdz) \rba ^\gamma \] \ala
&& +\, c_\gamma \E\[ \lba \int_0^{\al\wedge\tau_\eta}\int_{C(0,1)} |G_0(M_{s-},z)-G_0(0,z)| \,dz/z^2ds \rba ^\gamma \] \ala
&\le& c_\gamma \E\[ \int_0^{\al\wedge\tau_\eta}\int_{C(0,1)} |G_0(M_{s-},z)-G_0(0,z)|^\gamma \,dz/z^2ds \] \ala
&\le& c_\gamma \E\[ \int_0^{\al\wedge\tau_\eta} |M_{s}|^\gamma \,ds \] \le c_\gamma \int_0^\al \E[|M_{s\wedge\tau_\eta}|^\gamma]\,ds \le c_\gamma \al^2, \nonumber
\end{eqnarray}
where we used again Lemma \ref{nonss0} and Lemma \ref{momentofjump0}. Repeating the computations \eqref{cheby*}, \eqref{cheby2*} and using $\gamma\in(\beta(0)+\eta, 1\wedge 2\beta(0))$ yield the result.
\end{proof}

\section*{Acknowledgements}
This work is part of my phd thesis. I wish to thank my advisors St\'ephane Jaffard and St\'ephane Seuret for suggesting this study and for their constant support during the preparation of this paper. I also would like to thank Nicolas Fournier and Arnaud Durand for many stimulating  discussions.   Finally, I appreciate many valuable remarks of the anonymous referee which improve a lot the presentation and weaken some conditions.

%---------------------------------------------------------------------------
% Bibliographie
%---------------------------------------------------------------------------
\bibliographystyle{plain}
\bibliography{xyangbiblio}

\begin{thebibliography}{10}

\bibitem{abry2015irregularities}
Patrice Abry, Herwig Jaffard, and Herwig Wendt.
\newblock Irregularities and scaling in signal and image processing:
  multifractal analysis.
\newblock {\em Benoit Mandelbrot: A Life in Many Dimensions. Edited by FRAME
  MICHAEL ET AL. Published by World Scientific Publishing Co. Pte. Ltd., 2015.
  ISBN\# 9789814366076, pp. 31-116}, 1:31--116, 2015.

\bibitem{amin1993}
Kaushik~I Amin.
\newblock Jump diffusion option valuation in discrete time.
\newblock {\em The journal of finance}, 48(5):1833--1863, 1993.

\bibitem{applebaum2009}
David Applebaum.
\newblock {\em L\'evy processes and stochastic calculus}, volume 116 of {\em
  Cambridge Studies in Advanced Mathematics}.
\newblock Cambridge University Press, Cambridge, second edition, 2009.

\bibitem{ayache2005multifractional}
Antoine Ayache and Murad~S. Taqqu.
\newblock Multifractional processes with random exponent.
\newblock {\em Publ. Mat.}, 49(2):459--486, 2005.

\bibitem{balanca2015mbm}
Paul Balan{\c{c}}a.
\newblock Some sample path properties of multifractional {B}rownian motion.
\newblock {\em Stochastic Process. Appl.}, 125(10):3823--3850, 2015.

\bibitem{balanca2014levy}
Paul Balan\c{c}a.
\newblock Fine regularity of {L}\'evy processes and linear (multi)fractional
  stable motion.
\newblock {\em Electron. J. Probab.}, 19:no. 101, 37, 2014.

\bibitem{balanca2016}
Paul Balan\c{c}a and Leonid Mytnik.
\newblock Singularities of stable super-brownian motion.
\newblock {\em Arxiv, e-print}, 2016.

\bibitem{barczy2015}
M{\'a}ty{\'a}s Barczy, Zenghu Li, and Gyula Pap.
\newblock Yamada-{W}atanabe results for stochastic differential equations with
  jumps.
\newblock {\em Int. J. Stoch. Anal.}, pages Art. ID 460472, 23, 2015.

\bibitem{barral2013localformalism}
Julien Barral, Arnaud Durand, St{\'e}phane Jaffard, and St{\'e}phane Seuret.
\newblock Local multifractal analysis.
\newblock In {\em Fractal geometry and dynamical systems in pure and applied
  mathematics. {II}. {F}ractals in applied mathematics}, volume 601 of {\em
  Contemp. Math.}, pages 31--64. Amer. Math. Soc., Providence, RI, 2013.

\bibitem{barral2010increasing}
Julien Barral, Nicolas Fournier, St{\'e}phane Jaffard, and St{\'e}phane Seuret.
\newblock A pure jump {M}arkov process with a random singularity spectrum.
\newblock {\em Ann. Probab.}, 38(5):1924--1946, 2010.

\bibitem{barral2007timechange}
Julien Barral and St{\'e}phane Seuret.
\newblock The singularity spectrum of {L}\'evy processes in multifractal time.
\newblock {\em Adv. Math.}, 214(1):437--468, 2007.

\bibitem{barral2011localized}
Julien Barral and St{\'e}phane Seuret.
\newblock A localized {J}arn\'\i k-{B}esicovitch theorem.
\newblock {\em Adv. Math.}, 226(4):3191--3215, 2011.

\bibitem{bass1988uniqueness}
Richard~F. Bass.
\newblock Uniqueness in law for pure jump {M}arkov processes.
\newblock {\em Probab. Theory Related Fields}, 79(2):271--287, 1988.

\bibitem{bass2003stableSDE}
Richard~F. Bass.
\newblock Stochastic differential equations driven by symmetric stable
  processes.
\newblock In {\em S\'eminaire de {P}robabilit\'es, {XXXVI}}, volume 1801 of
  {\em Lecture Notes in Math.}, pages 302--313. Springer, Berlin, 2003.

\bibitem{bertoin1994integraltestforcovering}
Jean Bertoin.
\newblock On nowhere differentiability for {L}\'evy processes.
\newblock {\em Stochastics Stochastics Rep.}, 50(3-4):205--210, 1994.

\bibitem{chudley1961}
C~T Chudley and R~J Elliott.
\newblock Neutron scattering from a liquid on a jump diffusion model.
\newblock {\em Proceedings of the Physical Society}, 77(2):353, 1961.

\bibitem{cinlar1981semimartingaleMarkov}
E.~{\c{C}}inlar and J.~Jacod.
\newblock Representation of semimartingale {M}arkov processes in terms of
  {W}iener processes and {P}oisson random measures.
\newblock In {\em Seminar on {S}tochastic {P}rocesses, 1981 ({E}vanston,
  {I}ll., 1981)}, volume~1 of {\em Progr. Prob. Statist.}, pages 159--242.
  Birkh\"auser, Boston, Mass., 1981.

\bibitem{doring2012pssmp}
Leif D{\"o}ring and M{\'a}ty{\'a}s Barczy.
\newblock A jump type {SDE} approach to positive self-similar {M}arkov
  processes.
\newblock {\em Electron. J. Probab.}, 17:no. 94, 39, 2012.

\bibitem{durand2008treeindexedchain}
Arnaud Durand.
\newblock Random wavelet series based on a tree-indexed {M}arkov chain.
\newblock {\em Comm. Math. Phys.}, 283(2):451--477, 2008.

\bibitem{durand2009singularity}
Arnaud Durand.
\newblock Singularity sets of {L}\'evy processes.
\newblock {\em Probab. Theory Related Fields}, 143(3-4):517--544, 2009.

\bibitem{durand2012levyfields}
Arnaud Durand and St{\'e}phane Jaffard.
\newblock Multifractal analysis of {L}\'evy fields.
\newblock {\em Probab. Theory Related Fields}, 153(1-2):45--96, 2012.

\bibitem{dvoretzky1963irregularityofBM}
Aryeh Dvoretzky.
\newblock On the oscillation of the {B}rownian motion process.
\newblock {\em Israel J. Math.}, 1:212--214, 1963.

\bibitem{falconer2003}
Kenneth Falconer.
\newblock {\em Fractal geometry}.
\newblock John Wiley \& Sons, Inc., Hoboken, NJ, second edition, 2003.
\newblock Mathematical foundations and applications.

\bibitem{falconer2003tangentprocesses}
Kenneth~J. Falconer.
\newblock The local structure of random processes.
\newblock {\em J. London Math. Soc. (2)}, 67(3):657--672, 2003.

\bibitem{fournier2013pathuniquenessofstableSDE}
Nicolas Fournier.
\newblock On pathwise uniqueness for stochastic differential equations driven
  by stable {L}\'evy processes.
\newblock {\em Ann. Inst. Henri Poincar\'e Probab. Stat.}, 49(1):138--159,
  2013.

\bibitem{fu2010positiveSDE}
Zongfei Fu and Zenghu Li.
\newblock Stochastic equations of non-negative processes with jumps.
\newblock {\em Stochastic Process. Appl.}, 120(3):306--330, 2010.

\bibitem{herbin2006}
Erick Herbin.
\newblock From {$N$} parameter fractional {B}rownian motions to {$N$} parameter
  multifractional {B}rownian motions.
\newblock {\em Rocky Mountain J. Math.}, 36(4):1249--1284, 2006.

\bibitem{ikeda1989}
Nobuyuki Ikeda and Shinzo Watanabe.
\newblock {\em Stochastic differential equations and diffusion processes},
  volume~24 of {\em North-Holland Mathematical Library}.
\newblock North-Holland Publishing Co., Amsterdam; Kodansha, Ltd., Tokyo,
  second edition, 1989.

\bibitem{jacod2003}
Jean Jacod and Albert~N. Shiryaev.
\newblock {\em Limit theorems for stochastic processes}, volume 288 of {\em
  Grundlehren der Mathematischen Wissenschaften [Fundamental Principles of
  Mathematical Sciences]}.
\newblock Springer-Verlag, Berlin, second edition, 2003.

\bibitem{jaffard1997old}
St{\'e}phane Jaffard.
\newblock Old friends revisited: the multifractal nature of some classical
  functions.
\newblock {\em J. Fourier Anal. Appl.}, 3(1):1--22, 1997.

\bibitem{jaffard1999levy}
St{\'e}phane Jaffard.
\newblock The multifractal nature of {L}\'evy processes.
\newblock {\em Probab. Theory Related Fields}, 114(2):207--227, 1999.

\bibitem{jaffard2004survey}
St{\'e}phane Jaffard.
\newblock Wavelet techniques in multifractal analysis.
\newblock In {\em Fractal geometry and applications: a jubilee of {B}eno\^\i t
  {M}andelbrot, {P}art 2}, volume~72 of {\em Proc. Sympos. Pure Math.}, pages
  91--151. Amer. Math. Soc., Providence, RI, 2004.

\bibitem{khoshnevisan2000fastforfBm}
Davar Khoshnevisan and Zhan Shi.
\newblock Fast sets and points for fractional {B}rownian motion.
\newblock In {\em S\'eminaire de {P}robabilit\'es, {XXXIV}}, volume 1729 of
  {\em Lecture Notes in Math.}, pages 393--416. Springer, Berlin, 2000.

\bibitem{kurtz2011}
Thomas~G. Kurtz.
\newblock Equivalence of stochastic equations and martingale problems.
\newblock In {\em Stochastic analysis 2010}, pages 113--130. Springer,
  Heidelberg, 2011.

\bibitem{li2011strongsolution}
Zenghu Li and Leonid Mytnik.
\newblock Strong solutions for stochastic differential equations with jumps.
\newblock {\em Ann. Inst. Henri Poincar\'e Probab. Stat.}, 47(4):1055--1067,
  2011.

\bibitem{mytnik2015superstable}
Leonid Mytnik and Vitali Wachtel.
\newblock Multifractal analysis of superprocesses with stable branching in
  dimension one.
\newblock {\em Ann. Probab.}, 43(5):2763--2809, 2015.

\bibitem{orey1974fast}
Steven Orey and S.~James Taylor.
\newblock How often on a {B}rownian path does the law of iterated logarithm
  fail?
\newblock {\em Proc. London Math. Soc. (3)}, 28:174--192, 1974.

\bibitem{perkins1983slow}
Edwin Perkins.
\newblock On the {H}ausdorff dimension of the {B}rownian slow points.
\newblock {\em Z. Wahrsch. Verw. Gebiete}, 64(3):369--399, 1983.

\bibitem{perkins1998superbrownian}
Edwin~A. Perkins and S.~James Taylor.
\newblock The multifractal structure of super-{B}rownian motion.
\newblock {\em Ann. Inst. H. Poincar\'e Probab. Statist.}, 34(1):97--138, 1998.

\bibitem{revuzyor1999}
Daniel Revuz and Marc Yor.
\newblock {\em Continuous martingales and {B}rownian motion}, volume 293 of
  {\em Grundlehren der Mathematischen Wissenschaften [Fundamental Principles of
  Mathematical Sciences]}.
\newblock Springer-Verlag, Berlin, third edition, 1999.

\bibitem{sato2013}
Ken-iti Sato.
\newblock {\em L\'evy processes and infinitely divisible distributions},
  volume~68 of {\em Cambridge Studies in Advanced Mathematics}.
\newblock Cambridge University Press, Cambridge, 2013.
\newblock Translated from the 1990 Japanese original, Revised edition of the
  1999 English translation.

\bibitem{shepp1972randomcovering}
L.~A. Shepp.
\newblock Covering the line with random intervals.
\newblock {\em Z. Wahrscheinlichkeitstheorie und Verw. Gebiete}, 23:163--170,
  1972.

\bibitem{situ2005}
Rong Situ.
\newblock {\em Theory of stochastic differential equations with jumps and
  applications}.
\newblock Mathematical and Analytical Techniques with Applications to
  Engineering. Springer, New York, 2005.
\newblock Mathematical and analytical techniques with applications to
  engineering.

\bibitem{yang2015range}
Xiaochuan Yang.
\newblock Hausdorff dimension of the range and the graph of stable-like
  processes.
\newblock {\em Arxiv, e-print}, 2015.

\bibitem{yang2016thesis}
Xiaochuan Yang.
\newblock Dimension study of the regularity of jump diffusion processes.
\newblock {\em Phd Thesis}, 2016.

\end{thebibliography}

\end{document}